\documentclass[12pt]{article}

\usepackage{amssymb,a4}
\usepackage{amsmath,amsfonts,amssymb,amsthm, mathrsfs}

\newcommand{\1}{\mathbf {1}}
\newcommand{\Z}{{\mathbb Z}}
\newcommand{\C}{{\mathbb C}}

\newcommand{\wh}{{\widehat{\mathfrak h}}}

\newcommand{\mraff}{\mathrm{aff}}

\newtheorem{thm}{Theorem}[section]
\newtheorem{prop}[thm]{Proposition}
\newtheorem{lem}[thm]{Lemma}

\newtheorem{rmk}[thm]{Remark}
\newtheorem{definition}[thm]{Definition}

\begin{document}

\begin{center}
{\Large \bf  Fusion rules for $\mathbb{Z}_{2}$-orbifolds of affine and parafermion vertex operator algebras}

\end{center}

\begin{center}
{ Cuipo Jiang$^{a}$\footnote{Supported by China NSF grants No.11771281 and No.11531004.}
and Qing Wang$^{b}$\footnote{Supported by
China NSF grants No.11622107 and No.11531004, Natural Science Foundation of Fujian Province
	No.2016J06002.}\\
$\mbox{}^{a}$ School of Mathematical Sciences, Shanghai Jiao Tong University, Shanghai 200240, China\\
\vspace{.1cm}
$\mbox{}^{b}$ School of Mathematical Sciences, Xiamen University,
Xiamen 361005, China\\
}
\end{center}

\begin{abstract}
This paper is about the orbifold theory of affine and parafermion vertex operator algebras. It is known that the parafermion vertex operator algebra $K(sl_2,k)$ associated to the integrable highest weight modules for the affine Kac-Moody algebra $A_1^{(1)}$ is the building block of the general parafermion vertex operator $K(\mathfrak{g},k)$ for any finite dimensional simple Lie algebra $\mathfrak{g}$ and any positive integer $k$.
We first classify the irreducible modules of $\Z_{2}$-orbifold of the simple affine vertex operator algebra of type
$A_1^{(1)}$ and determine their fusion rules. Then we study the representations of the $\Z_{2}$-orbifold of the parafermion vertex operator algebra $K(sl_2,k)$, we give the quantum dimensions, and more technically, fusion rules for the $\mathbb{Z}_{2}$-orbifold of the parafermion vertex operator algebra $K(sl_2,k)$ are  completely determined.

\end{abstract}

\section{Introduction}
\def\theequation{1.\arabic{equation}}
\setcounter{equation}{0}

This paper is a continuation in a series of papers on the study of the orbifold theory of affine and parafermion vertex operator algebras. It is known that the parafermion vertex operator algebra $K(\mathfrak{g},k)$ is the commutant of a Heisenberg vertex operator subalgebra in the simple affine vertex operator algebra $L_{\hat{\mathfrak{g}}}(k,0)$, where $L_{\hat{\mathfrak{g}}}(k,0)$ is the integrable highest weight module with the positive integer level $k$ for the affine Kac-Moody algebra $\hat{\mathfrak{g}}$ associated to a finite dimensional simple Lie algebra $\mathfrak{g}$ over ${\mathbb C}$. We denote $K(sl_{2},k)$ by $K_{0}$ and $L_{\hat{sl_{2}}}(k,0)$ by $L(k,0)$ in this paper. Since parafermion vertex operator algebras can be identified with $W$-algebras \cite{DLY2}, the orbifold theory of the parafermion vertex algebras corresponds to the orbifold theory of $W$-algebras. Some conjectures in the physics literature about the orbifold $W$-algebras have been studied and solved in \cite{AL}, \cite{ACKL}, \cite{KL}. These results about the orbifold $W$-algebras are mainly structural aspects. Our interest is to study the representation theory of the orbifold parafermion vertex operator algebra from the point of vertex algebras. From \cite{DLY2}, we know that the full automorphism group of the parafermion vertex operator algebra $K_{0}$ for $k\geq 3$ is the
group of order $2$ generated by the automorphism $\sigma$, which is determined by $\sigma(h)=-h$, $\sigma(e)=f$, $\sigma(f)=e$, where $\{ h, e, f\}$ is a standard Chevalley basis of $sl_2$ with brackets $[h,e] = 2e$, $[h,f] = -2f$ and $[e,f] = h$. We have classified the irreducible modules of the orbifold parafermion vertex operator algebra $K_{0}^{\sigma}$ in \cite{JW}, where $K_{0}^{\sigma}$ is the fixed-point vertex operator subalgebra of $K_0$ under $\sigma$. A natural problem next is to determine the fusion rules for $K_{0}^{\sigma}$. Note that the vertex operator algebra $K_{0}^{\sigma}$ can be viewed as a subalgebra of the orbifold affine vertex operator algebra $L(k,0)^{\sigma}$, where $L(k,0)^{\sigma}$ is the fixed-point vertex operator subalgebra of $L(k,0)$ under $\sigma$. In order to understand the representation theory of the orbifold parafermion vertex operator algebra $K_{0}^{\sigma}$ better, we should first understand the representation theory of the orbifold affine vertex operator algebra $L(k,0)^{\sigma}$ first. For this purpose, we classify the irreducible modules of $L(k,0)^{\sigma}$ and determine the fusion rules for $L(k,0)^{\sigma}$ in Section 3. We obtain  Theorem \ref{thm:orbifold3} that there are two kinds of irreducible modules for $L(k,0)^{\sigma}$. One kind is the untwisted type modules coming from the irreducible $L(k,0)$-modules, and the other kind is the twisted type modules coming from the $\sigma$-twisted $L(k,0)$-modules.
Furthermore, we determine the contragredient modules of all these irreducible $L(k,0)^{\sigma}$-modules in Theorem \ref{thm:contragredient}. These results together with the symmetric property of fusion rules imply that we only need to determine two kinds of fusion products, one is the fusion product between the untwisted type modules and the untwisted type modules, and the other is the fusion product between the untwisted type modules and the twisted type modules. Our first step is to construct the intertwining operators among untwisted and twisted $L(k,0)$-modules. We use the $\Delta$-operator introduced by Li in \cite{L3}. Then the fusion products between the untwisted type modules and the twisted type modules can be obtained by applying the fusion rules for the affine vertex operator algebra $L(k,0)$ and the intertwining operator constructed from the $\Delta$-operator. Furthermore, by observing the action of the automorphism $\sigma$ on the $\Delta$-operator, the fusion products between the untwisted type modules and the untwisted type modules follow from the fusion products between the untwisted type modules and the twisted type modules.

The determination of the fusion rules for $K_{0}^{\sigma}$ is much more complicated. We first determine the quantum dimensions of the irreducible $K_{0}^{\sigma}$-modules, which can help us to determine the fusion rules for $K_{0}^{\sigma}$. However it is far from the complete determination of the fusion rules for $K_{0}^{\sigma}$. Our strategy is to employ the lattice realization of the irreducible $K_0$-modules \cite{DLY2} and  the lowest weights of the irreducible $K_{0}^{\sigma}$-modules \cite{JW}, together with the decomposition of the irreducible $L(k,0)$-modules $L(k,i)$ viewed as the modules of the lattice vertex operator subalgebra $V_{{\Z}\gamma}\subseteq L(k,0)$\cite{DLY2} for $0\leq i\leq k$. From the classification results of the irreducible modules of $K_{0}^{\sigma}$, there are two families of untwisted type $K_{0}^{\sigma}$-modules. One family is from the irreducible modules of $K_0$, which are not irreducible as $K_{0}^{\sigma}$-modules. We call it the untwisted module of type $I$. The other family is from the irreducible modules of $K_0$, which are also irreducible as $K_{0}^{\sigma}$-modules. We call it the untwisted module of type $II$. We would like to point out that the main difficulty to determine the fusion products between the untwisted type modules and the untwisted type modules of $K_{0}^{\sigma}$ is to find which one of the irreducible $K_{0}^{\sigma}$-modules of type $I$ can survive in the decomposition of the fusion product, and to distinguish the inequivalent modules emerging in the decomposition of the fusion product. The fusion products between the untwisted type modules and the twisted type modules of $K_{0}^{\sigma}$ are extremely complicated in the case that the level $k$ is even, because from \cite{JW}, we know that in the level $\frac{k}{2}$, there are two irreducible twisted modules of $K_0$, and the lowest weight vector can be in the grade zero or in the grade $\frac{1}{2}$ of the $\sigma$-twisted module of $K_0$. Thus as the $K_{0}^{\sigma}$-modules, there are four irreducible modules in the level $\frac{k}{2}$, when it emerges in the decomposition of the fusion product between the untwisted type module and the twisted type module of $K_{0}^{\sigma}$. We need to distinguish which one can survive for certain cases. The strategy is that we come back to the lattice realization of the irreducible $K_0$-modules $M^{i,j}$ for $0\leq i\leq k$, $0\leq j\leq i$\cite{DLY2}, and we technically  use another basis of the Lie algebra $sl_2$ and apply the intertwining operator among the modules of the lattice vertex operator algebra, together with the analysis of the lowest weights of the irreducible $K_{0}^{\sigma}$-modules we obtained in \cite{JW}. Furthermore, we determine the contragredient modules of all the irreducible $K_0^{\sigma}$-modules, thus the fusion rules for $K_{0}^{\sigma}$ are completely determined.

The paper is organized as follows. In Section 2, we recall some results about the parafermion vertex operator algebra $K_0$, its orbifold vertex operator subalgebra $K_{0}^{\sigma}$ and their irreducible modules. In Section 3, we classify the irreducible modules of the $\Z_{2}$-orbifold $L(k,0)^{\sigma}$ of the affine vertex operator algebra $L(k,0)$ and determine the fusion rules for $L(k,0)^{\sigma}$. In Section 4, we give the quantum dimensions for irreducible $K_{0}^{\sigma}$-modules. In Section 5, we determine the fusion rules for the $\Z_{2}$-orbifold of parafermion vertex operator algebra $K_{0}$.

\section{Preliminaries}
\label{Sect:V(k,0)}\def\theequation{2.\arabic{equation}}

In this section, we recall from \cite{DLY2}, \cite{DLWY}, \cite{DW2}, \cite{ALY1} and \cite{JW} some basic results on the parafermion vertex
operator algebra associated to the
irreducible highest weight module for the affine Kac-Moody algebra
$A_1^{(1)}$ of level $k$ with $k$ being a positive integer and their $\Z_{2}$-orbifolds. We first recall the notion of the parafermion vertex operator algebra.

We are working in the setting of \cite{DLY2}. Let $\{ h, e, f\}$
be a standard Chevalley basis of $sl_2$ with Lie brackets $[h,e] = 2e$, $[h,f] = -2f$,
$[e,f] = h$ and the normalized Killing form $\langle h, h \rangle=2$, $\langle e, f \rangle=1$, $\langle h, e \rangle=\langle h,f \rangle=\langle e,e \rangle=\langle f,f \rangle=0$. Let $\widehat{sl}_2 = sl_2 \otimes \C[t,t^{-1}]
\oplus \C C$ be the affine Lie algebra associated to $sl_2$. Let $k \ge 1$
be an integer and
\begin{equation*}
V(k,0) = V_{\widehat{sl}_2}(k,0) = \mbox{Ind}_{sl_2 \otimes \C[t]\oplus \C
C}^{\widehat{sl}_2}\C
\end{equation*}
be the induced $\widehat{sl}_2$-module such that $sl_2 \otimes \C[t]$ acts
as $0$ and $C$ acts as $k$ on $\mathbf{1}=1$. Then $V(k,0)$ is a
vertex operator algebra generated by $a(-1)\1$ for $a\in sl_2$ such
that
$$Y(a(-1)\1,z) = a(z)=\sum_{n \in \Z} a(n)z^{-n-1}$$
where $a(n)=a\otimes t^n$, with the
vacuum vector $\1$ and the Virasoro vector
\begin{align*}
\omega_{\mraff} &= \frac{1}{2(k+2)} \Big( \frac{1}{2}h(-1)^2\1 +
e(-1)f(-1)\1 + f(-1)e(-1)\1 \Big)\\
&= \frac{1}{2(k+2)} \Big( -h(-2)\1 + \frac{1}{2}h(-1)^2\1 + 2e(-1)f(-1)\1
\Big)
\end{align*}
of central charge $\frac{3k}{k+2}$ (e.g. \cite{FZ},
\cite{Kac}, \cite[Section 6.2]{LL}).

Let $M(k)$ be the vertex operator subalgebra of $V(k,0)$
generated by $h(-1)\1$ with the Virasoro
element
$$\omega_{\gamma} = \frac{1}{4k}
h(-1)^{2}\1$$
of central charge $1$.


The vertex operator algebra $V(k,0)$ has a unique maximal ideal $\mathcal{J}$, which is generated by a weight $k+1$ vector $e(-1)^{k+1}\1$ \cite{Kac}. The quotient algebra $L(k,0)=V(k,0)/\mathcal{J}$ is a
simple, rational  vertex operator algebra as $k$ is a positive
integer (cf. \cite{FZ}, \cite{LL}). Moreover, the image of $M(k)$
in $L(k,0)$ is isomorphic to $M(k)$ and will be
denoted by $M(k)$ again. Set
\begin{equation*}
 K(sl_2,k)=\{v \in L(k,0)\,|\, h(m)v =0
\text{ for }\; h\in {\mathfrak h},
 m \ge 0\}.
\end{equation*}
Then $K(sl_2,k)$ which is the space of highest weight vectors
with highest weight $0$ for $\wh$
is the commutant of $M(k)$ in $L(k,0)$
and is called the parafermion vertex operator algebra associated
to the irreducible highest weight module $L(k,0)$ for
$\widehat{sl_2}.$ The Virasoro element of $K(sl_2,k)$ is given by
$$\omega =\omega_{\mraff} - \omega_{\gamma}=\frac{1}{2k(k+2)} \Big(-kh(-2)\1-h(-1)^2\1+2k
e(-1)f(-1)\1 \Big)
$$
with central charge $\frac{2(k-1)}{k+2}$,
where we still use $\omega_{\mraff}, \omega_{\gamma}$ to
denote their images in $L(k,0)$. We denote $K(sl_2,k)$ by $K_0$.

Set \begin{equation*}\label{eq:W3}
\begin{split}
W^3 &= k^2 h(-3)\1 + 3 k h(-2)h(-1)\1 +
2h(-1)^3\1 - 6k h(-1)e(-1)f(-1)\1 \\
& \quad + 3 k^2e(-2)f(-1)\1 - 3 k^2e(-1)f(-2)\1
\end{split}
\end{equation*}
in $V(k,0)$, and also denote its image in $L(k,0)$ by $W^{3}$. It was proved in \cite{DLY2}(cf.\cite{DLWY}, \cite{DW1}) that the parafermion vertex operator algebra $K_0$ is simple and is generated by $\omega$ and $W^{3}$. If $k\geq3$, the parafermion vertex operator algebra $K_0$ in fact is generated by $W^{3}$. The irreducible $K_0$-modules $M^{i,j}$ for $0\leq i\leq k, 0\leq j\leq k-1$ were constructed in \cite{DLY2}. Note that $K_0=M^{0,0}$. It was also proved in
\cite[Theorem 4.4]{DLY2} that $M^{i,j}\cong M^{k-i,k-i+j}$ as $K_0$-module .   Theorem 8.2 in \cite{ALY1} showed that the $\frac{k(k+1)}{2}$ irreducible $K_0$-modules $M^{i,j}$ for $1\leq i\leq k, 0\leq j\leq i-1$ constructed in \cite{DLY2} form a complete set of isomorphism classes of irreducible $K_0$-modules. Moreover, $K_0$ is $C_2$-cofinite \cite{ALY1} and rational \cite{ALY2} (see also \cite{DR}).

 Let $L(k,i)$ for $0\leq i\leq k$ be the irreducible modules for the rational vertex operator algebra $L(k,0)$ with the top level $U^{i}=\bigoplus_{j=0}^{i}\mathbb{C}v^{i,j}$ which is an $(i+1)$-dimensional irreducible module of the simple Lie algebra $\C h(0)\oplus\C e(0)\oplus \C f(0)\cong sl_2$. The top level of $M^{i,j}$ is a one dimensional space spanned by $v^{i,j}$ for $0\leq i\leq k, 0\leq j\leq i$\cite{DLY2}. The following result was due to \cite{DLY2}.

\begin{lem}\label{lem:lowest}
The operator $o(\omega)=\omega_{1}$ acts on $v^{i,j}, \ 0\leq i\leq k,\ 0\leq j\leq i$ as follows:

\begin{eqnarray}
o(\omega)v^{i,j}=\frac{1}{2k(k+2)}\Big(k(i-2j)-(i-2j)^{2}+2kj(i-j+1)\Big)v^{i,j}.
 \end{eqnarray}
\end{lem}

Let $\sigma$ be an automorphism of Lie algebra $sl_2$ defined by $\sigma(h)=-h, \ \sigma(e)=f,\ \sigma(f)=e$. $\sigma$ can be lifted to an automorphism $\sigma$ of the vertex operator algebra $V(k,0)$ of order 2 in the following way:
$$\sigma(x_{1}(-n_{1})\cdots x_{s}(-n_{s})\1)=\sigma(x_{1})(-n_{1})\cdots \sigma(x_{s})(-n_{s})\1$$
for $x_{i}\in sl_2$ and $n_{i}>0$. Then $\sigma$ induces an automorphism of $L(k,0)$ as $\sigma$ preserves the unique maximal ideal $\mathcal{J}$, and the Virasoro element $\omega_{\gamma}$ is invariant under $\sigma$. Thus $\sigma$ induces an automorphism of the parafermion vertex operator algebra $K_{0}$. In fact, $\sigma(\omega)=\omega,\ \sigma(W^{3})=-W^{3}$.

\begin{lem}\cite{DLY2}\label{lem:auto}
If $k\geq 3$, the automorphism group Aut$K_{0}=\langle \sigma \rangle$ is of order 2.
\end{lem}
\begin{rmk} If $k=1$, $K_{0}=\C \1$. If $k=2$, $K_0$ is generated by $\omega$. Thus the automorphism group Aut$K_0=\{1\}$ is trivial for $k=1$ and $k=2$. Therefore, by Lemma \ref{lem:auto},  we  only need to consider the orbifold of parafermion vertex operator algebra under the automorphism $\sigma$ for $k\geq 3$.
\end{rmk}

Let $K_{0}^{\sigma}$ be the ${\mathbb Z}_2$-orbifold vertex operator algebra, i.e., the fixed-point vertex operator subalgebra of $K_0$ under the automorphism $\sigma$. The following theorem gives the classification of the irreducible modules of $K_{0}^{\sigma}$ for $k\geq 3$ \cite{JW}.

 \begin{thm}\cite{JW}\label{thm:orbifold3'}
If $k=2n+1$, $n\geq 1$, there are $\frac{(k+1)(k+7)}{4}$ inequivalent irreducible modules of $K_{0}^{\sigma}$.
 If $k=2n$, $n\geq 2$, there are $\frac{(k^{2}+8k+28)}{4}$ inequivalent irreducible modules of $K_{0}^{\sigma}$. More precisely, if $k=2n+1$, $n\geq 1$, the set
 \begin{eqnarray*}
&& \{ W(k,i)^{j} \ \mbox{for} \  0\leq i\leq \frac{k-1}{2}, j=1,2,\\
&&(M^{i,j})^{s} \ \mbox{for} \  (i,j)=(i,\frac{i}{2}), i=2,4,6,\cdots,2n, \ \mbox{and} \ (i,j)=(2n+1,0), s=0,1,\\ && M^{i,0} \ \mbox{for} \ 1\leq i\leq \frac{k-1}{2},
  M^{i,j} \ \mbox{for} \ 3\leq i\leq k, \mbox{if} \  i=2m, 1\leq j\leq m-1, \mbox{if} \ i=2m+1, 1\leq j\leq m\} \\
 \end{eqnarray*} gives all inequivalent irreducible $K_{0}^{\sigma}$-modules. If $k=2n$, $n\geq 2$, the set
 \begin{eqnarray*}
&& \{ W(k,i)^{j} \ \mbox{for} \  0\leq i\leq \frac{k}{2}, j=1,2,  \widetilde{W(k,\frac{k}{2})}^{j} \ \mbox{for} \ j=1,2,\\
&&(M^{i,j})^{s} \ \mbox{for} \  (i,j)=(i,\frac{i}{2}), i=2,4,6,\cdots,2n, (i,j)=(n,0) \mbox{and} \ (i,j)=(2n,0), s=0,1,\\ && M^{i,0} \ \mbox{for} \ 1\leq i\leq \frac{k-2}{2},
  M^{i,j} \ \mbox{for} \ 3\leq i\leq k, \mbox{if} \  i=2m, 1\leq j\leq m-1, \mbox{if} \ i=2m+1, 1\leq j\leq m\} \\
 \end{eqnarray*} gives all inequivalent irreducible $K_{0}^{\sigma}$-modules.
\end{thm}

\begin{rmk}\label{rmk:orbifold3'} With the notations in  Theorem \ref{thm:orbifold3'}, we call $W(k,i)^{j}$ and $\widetilde{W(k,\frac{k}{2})}^{j}$ twisted type modules and  $(M^{i,j})^{s}, M^{i,j}$ untwisted modules of type $I$ and type $II$ respectively.

\end{rmk}

\section{Fusion rules for the $\Z_{2}$-orbifold of the affine vertex operator algebra $L(k,0)$
}\label{Sect:affine case}\def\theequation{3.\arabic{equation}}

In this section, we first recall the definition of weak $g$-twisted modules, $g$-twisted modules and admissible $g$-twisted modules following \cite{DLM3, DLM4}. Let $L(k,0)^{\sigma}$ be the $\Z_{2}$-orbifold vertex operator subalgebra of the affine vertex operator algebra $L(k,0)$, i.e., the fixed-point subalgebra of $L(k,0)$ under $\sigma$. We then classify and construct the irreducible modules for $L(k,0)^{\sigma}$. Furthermore, we determine the contragredient modules of irreducible $L(k,0)^{\sigma}$-modules and the fusion rules for the vertex operator algebra $L(k,0)^{\sigma}$.

Let $\left(V,Y,1,\omega\right)$ be a vertex operator algebra (see
\cite{FLM}, \cite{LL}) and $g$ an automorphism of $V$ with finite order $T$.  Let $W\left\{ z\right\} $
denote the space of $W$-valued formal series in arbitrary complex
powers of $z$ for a vector space $W$. Denote the decomposition of $V$ into eigenspaces with respect to the action of $g$ by $$V=\bigoplus_{r\in\Z}V^{r},$$
where $V^{r}=\{v\in V|\ gv=e^{-\frac{2\pi ir}{T}}v\},$ $i=\sqrt{-1}$.

\begin{definition}A \emph{weak $g$-twisted $V$-module} $M$ is
a vector space with a linear map
\[
Y_{M}:V\to\left(\text{End}M\right)\{z\}
\]

\[
v\mapsto Y_{M}\left(v,z\right)=\sum_{n\in\mathbb{Q}}v_{n}z^{-n-1}\ \left(v_{n}\in\mbox{End}M\right)
\]
which satisfies the following conditions for $0\leq r\leq T-1$, $u\in V^{r}\ ,v\in V, w\in M$:

\[ Y_{M}\left(u,z\right)=\sum_{n\in\frac{r}{T}+\mathbb{Z}}u_{n}z^{-n-1}
\]

\[
u_{n}w=0\ {\rm for} \ n\gg0,
\]

\[
Y_{M}\left(\mathbf{1},z\right)=Id_{M},
\]

\[
z_{0}^{-1}\text{\ensuremath{\delta}}\left(\frac{z_{1}-z_{2}}{z_{0}}\right)Y_{M}\left(u,z_{1}\right)
Y_{M}\left(v,z_{2}\right)-z_{0}^{-1}\delta\left(\frac{z_{2}-z_{1}}{-z_{0}}\right)Y_{M}\left(v,z_{2}\right)Y_{M}\left(u,z_{1}\right)
\]

\[
=z_{1}^{-1}\left(\frac{z_{2}+z_{0}}{z_{1}}\right)^{\frac{r}{T}}\delta\left(\frac{z_{2}+z_{0}}{z_{1}}\right)Y_{M}\left(Y\left(u,z_{0}\right)v,z_{2}\right),
\]
 where $\delta\left(z\right)=\sum_{n\in\mathbb{Z}}z^{n}$. \end{definition}
 The following identities are the consequences of the twisted-Jacobi identity \cite{DLM3} (see also \cite{Ab}, \cite{DJ}).
 \begin{eqnarray}[u_{m+\frac{r}{T}},v_{n+\frac{s}{T}}]=\sum_{i=0}^{\infty}\binom{m+\frac{r}{T}}{i} (u_{i}v)_{m+n+\frac{r+s}{T}-i},\label{eq:3.1.}\end{eqnarray}
 \begin{eqnarray}\sum_{i\geq 0}\binom{\frac{r}{T}}{i}(u_{m+i}v)_{n+\frac{r+s}{T}-i}=\sum_{i\geq 0}(-1)^{i}\binom{m}{i}(u_{m+\frac{r}{T}-i}v_{n+\frac{s}{T}+i}-(-1)^{m}v_{m+n+\frac{s}{T}-i}u_{\frac{r}{T}+i}),\label{eq:3.2.}\end{eqnarray}
 where $u\in V^{r}, \ v\in V^{s},\ m,n\in \Z$.

\begin{definition}

A \emph{$g$-twisted $V$-module} is a weak $g$-twisted $V$-module\emph{
}$M$ which carries a $\mathbb{C}$-grading $M=\bigoplus_{\lambda\in\mathbb{C}}M_{\lambda},$
where $M_{\lambda}=\{w\in M|L(0)w=\lambda w\}$ and $L(0)$ is one of the coefficient operators of $Y(\omega,z)=\sum_{n\in\mathbb{Z}}L(n)z^{-n-2}.$
Moreover we require
that $\dim M_{\lambda}$ is finite and for fixed $\lambda,$ $M_{\lambda+\frac{n}{T}}=0$
for all small enough integers $n.$

\end{definition}

\begin{definition}An \emph{admissible $g$-twisted $V$-module} $M=\oplus_{n\in\frac{1}{T}\mathbb{Z}_{+}}M\left(n\right)$
is a $\frac{1}{T}\mathbb{Z}_{+}$-graded weak $g$-twisted module
such that $u_{m}M\left(n\right)\subset M\left(\mbox{wt}u-m-1+n\right)$
for homogeneous $u\in V$ and $m,n\in\frac{1}{T}\mathbb{Z}.$ $ $

\end{definition}

If $g=Id_{V}$, we have the notions of weak, ordinary and admissible
$V$-modules \cite{DLM3}.

\begin{definition}A vertex operator algebra $V$ is called \emph{$g$-rational}
if the admissible $g$-twisted module category is semisimple. \end{definition}
\begin{rmk} Since $K_0$ is a rational vertex operator algebra, $K_{0}^{\sigma}$ is $C_2$-cofinite and rational \cite{M2}, \cite{CM}, \cite{CKLR}, and $K_{0}$ is $\sigma$-rational \cite{DH}.
\end{rmk}

The following lemma about $g$-rational vertex operator algebras is
well known \cite{DLM3}.

\begin{lem} If $V$ is  $g$-rational, then

(1) Any irreducible admissible $g$-twisted $V$-module $M$ is a $g$-twisted $V$-module, and there exists a $\lambda\in\mathbb{C}$
such that $M=\oplus_{n\in\frac{1}{T}\mathbb{Z}_{+}}M_{\lambda+n}$
where $M_{\lambda}\neq0.$ And $\lambda$ is called the conformal weight
of $M;$

(2) There are only finitely many irreducible admissible $g$-twisted
$V$-modules up to isomorphism. \end{lem}



Let $M=\bigoplus_{n\in\frac{1}{T}\mathbb{Z}_{+}}M(n)$ be an admissible $g$-twisted $V$-module, the contragredient module $M^{'}$ is defined as follows: $M'=\bigoplus_{n\in\frac{1}{T}\mathbb{Z}_{+}}M(n)^{*}$, where $M(n)^{*}=\mbox{Hom}_{\mathbb{C}}(M(n),\mathbb{C}).$
The vertex
operator $Y_{M'}(v,z)$ is defined for $v\in V$ via
\begin{eqnarray}\label{contragredient}
\langle Y_{M'}(v,z)f,u\rangle=\langle f,Y_{M}(e^{zL(1)}(-z^{-2})^{L(0)}v,z^{-1})u\rangle,\label{eq:3.0}
\end{eqnarray}
where $\langle f,w\rangle=f(w)$ is the natural paring $M'\times M\to\mathbb{C}.$

\begin{rmk} $(M^{'}, Y_{M^{'}})$ is an admissible $g^{-1}$-twisted $V$-module \cite{FHL}. One can also define the contragredient module $M^{'}$ for a $g$-twisted $V$-module $M$. In this case, $M^{'}$ is a $g^{-1}$-twisted $V$-module. Moreover, $M$ is irreducible if and only if $M^{'}$ is irreducible.
\end{rmk}





Now we  recall from \cite{FHL} the notions of intertwining operators and fusion rules.

\begin{definition} Let $(V,\ Y)$ be a vertex operator algebra and
let $(W^{1},\ Y^{1}),\ (W^{2},\ Y^{2})$ and $(W^{3},\ Y^{3})$ be
$V$-modules. An \emph{intertwining operator} of type $\left(\begin{array}{c}
W^{3}\\
W^{1\ }W^{2}
\end{array}\right)$ is a linear map
\[
I(\cdot,\ z):\ W^{1}\to\text{\ensuremath{\mbox{Hom}(W^{2},\ W^{3})\{z\}}}
\]

\[
u\to I(u,\ z)=\sum_{n\in\mathbb{Q}}u_{n}z^{-n-1}
\]
 satisfying:

(1) for any $u\in W^{1}$ and $v\in W^{2}$, $u_{n}v=0$ for $n$
sufficiently large;

(2) $I(L(-1)v,\ z)=\frac{d}{dz}I(v,\ z)$;

(3) (Jacobi identity) for any $u\in V,\ v\in W^{1}$

\[
z_{0}^{-1}\delta\left(\frac{z_{1}-z_{2}}{z_{0}}\right)Y^{3}(u,\ z_{1})I(v,\ z_{2})-z_{0}^{-1}\delta\left(\frac{-z_{2}+z_{1}}{z_{0}}\right)I(v,\ z_{2})Y^{2}(u,\ z_{1})
\]
\[
=z_{2}^{-1}\left(\frac{z_{1}-z_{0}}{z_{2}}\right)I(Y^{1}(u,\ z_{0})v,\ z_{2}).
\]

The space of all intertwining operators of type $\left(\begin{array}{c}
W^{3}\\
W^{1}\ W^{2}
\end{array}\right)$ is denoted by
$$I_{V}\left(\begin{array}{c}
W^{3}\\
W^{1}\ W^{2}
\end{array}\right).$$ Let $N_{W^{1},\ W^{2}}^{W^{3}}=\dim I_{V}\left(\begin{array}{c}
W^{3}\\
W^{1}\ W^{2}
\end{array}\right)$. These integers $N_{W^{1},\ W^{2}}^{W^{3}}$ are usually called the
\emph{fusion rules}. \end{definition}




\begin{definition} Let $V$ be a vertex operator algebra, and $W^{1},$
$W^{2}$ be two $V$-modules. A module $(W,I)$, where $I\in I_{V}\left(\begin{array}{c}
\ \ W\ \\
W^{1}\ \ W^{2}
\end{array}\right),$ is called a \emph{tensor product} (or fusion product) of $W^{1}$
and $W^{2}$ if for any $V$-module $M$ and $\mathcal{Y}\in I_{V}\left(\begin{array}{c}
\ \ M\ \\
W^{1}\ \ W^{2}
\end{array}\right),$ there is a unique $V$-module homomorphism $f:W\rightarrow M,$ such
that $\mathcal{Y}=f\circ I.$ As usual, we denote $(W,I)$ by $W^{1}\boxtimes_{V}W^{2}.$
\end{definition}

\begin{rmk}
It is well known that if $V$ is rational, then for any two irreducible
$V$-modules $W^{1}$ and $W^{2},$ the fusion product $W^{1}\boxtimes_{V}W^{2}$ exists and
$$
W^{1}\boxtimes_{V}W^{2}=\sum_{W}N_{W^{1},\ W^{2}}^{W}W,
$$
 where $W$ runs over the set of equivalence classes of irreducible
$V$-modules.
\end{rmk}
Fusion rules have the following symmetric property \cite{FHL}.

\begin{prop}\label{fusionsymm.}
Let $W^{i} (i=1,2,3)$ be $V$-modules. Then
$$N_{W^{1},W^{2}}^{W^{3}}=N_{W^{2},W^{1}}^{W^{3}}, \ N_{W^{1},W^{2}}^{W^{3}}=N_{W^{1},(W^{3})^{'}}^{(W^{2})^{'}}.$$
\end{prop}

We will use the following lemma from \cite{DL} later.

\begin{lem}\label{intertwining} Let $V$ be a vertex operator algebra, and let $W^{1}$ and $W^{2}$ be irreducible $V$-modules and $W^{3}$ a $V$-module. If $I$ is a nonzero intertwining operator of type $\left(\begin{array}{c}
W^{3}\\
W^{1}\ W^{2}
\end{array}\right)$, then $I(u,z)v\neq 0$ for any nonzero vectors $u\in W^{1}$ and $v\in W^{2}$.
\end{lem}

We fix some notations. Let $W^{1},W^{2},W^{3}$ be irreducible $L(k,0)^{\sigma}$-modules. In this section, we use $I\left(\begin{array}{c}W^3\\
W^{1}\,W^{2}\end{array}\right)$ to denote the space  $I_{L(k,0)^{\sigma}}\left(\begin{array}{c}W^3\\
W^{1}\,W^{2}\end{array}\right)$ of all intertwining operators
of type $\left(\begin{array}{c}W^3\\
W^{1}\,W^{2}\end{array}\right)$, and use $W^{1}\boxtimes W^{2}$ to denote the fusion product $W^{1}\boxtimes_{L(k,0)^{\sigma}}W^{2}$
for simplicity. We recall the fusion rules for the affine vertex operator algebra of type $A_1^{(1)}$ \cite{TK} for later use.


\begin{lem}\label{affine} $$L(k,i)\boxtimes_{L(k,0)} L(k,j)=\sum\limits_{l} L(k,l),$$
where $|i-j|\leq l\leq i+j, \ i+j+l\in 2\mathbb{Z},\ i+j+l\leq 2k.$\\
\end{lem}

We notice that since $L(k,0)$ is rational, $L(k,0)^{\sigma}$ is rational, and thus $L(k,0)$ is $\sigma$-rational. Then from \cite{DLM4}, we have the following result.

\begin{prop}\label{prop:twisted1}
 There are precisely $k+1$ inequivalent irreducible $\sigma$-twisted modules of $L(k,0)$.
\end{prop}

\begin{proof} Since $L(k,0)$ is $\sigma$-rational, from \cite{DLM4}, we know that the number of inequivalent irreducible $\sigma$-twisted modules of $L(k,0)$ is precisely the number of $\sigma$-stable irreducible untwisted modules of $L(k,0)$. Notice that $L(k,i)$ for $0\leq i\leq k$ exhaust all the irreducible modules for  $L(k,0)$ with the top level $U^{i}=\bigoplus_{j=0}^{i}\mathbb{C}v^{i,j}$. By direct calculation, we have
\begin{eqnarray}\label{lowest1}
o(\omega_{\mraff})v^{i,j}=\omega_{\mraff}(1)v^{i,j}=\frac{1}{2(k+2)}\Big(h(0)+\frac{1}{2}h(0)^{2}+2f(0)e(0)\Big)v^{i,j}=\frac{i(i+2)}{4(k+2)}v^{i,j}.
 \end{eqnarray}
We see that these lowest weights $\frac{i(i+2)}{4(k+2)}$ are pairwise different for $0\leq i\leq k$, which shows that $L(k,i)$ for $0\leq i\leq k$ are $\sigma$-stable irreducible modules. Thus there are totally $k+1$ inequivalent irreducible $\sigma$-twisted modules of $L(k,0)$.
\end{proof}

Recall from \cite{JW} that $\{ h, e, f\}$
is a standard Chevalley basis of $sl_2$ with brackets $[h,e] = 2e$, $[h,f] = -2f$,
$[e,f] = h$. Set $$h^{'}=e+f, \ e^{'}=\frac{1}{2}(h-e+f),\ f^{'}=\frac{1}{2}(h+e-f).$$ Then $\{ h^{'}, e^{'}, f^{'}\}$ is a $sl_2$-triple. Let $h^{''}=\frac{1}{4}h^{'}=\frac{1}{4}(e+f)$, and
$$\Delta(h^{''},z)=z^{h^{''}(0)}\mbox{exp}(\sum_{k=1}^{\infty}\frac{h^{''}(k)}{-k}(-z)^{-k}).$$
 Note that $L(k,i)$ for $0\leq i\leq k$ are all the irreducible modules for the rational vertex operator algebra $L(k,0)$. From \cite{L2}, we have the following result.

\begin{lem}\label{lem:twisted}
For $0\leq i\leq k$, $(\overline{L(k,i)}, Y_{\sigma}(\cdot,z))=(L(k,i),Y(\Delta(h^{''},z)\cdot,z))$ are irreducible $\sigma$-twisted $L(k,0)$-modules.
\end{lem}

As in \cite{JW}, for $u\in L(k,0)$ such that $\sigma(u)=e^{-\pi ri}u$, $i=\sqrt{-1}$, $r\in\Z$, we use the notation $u_{n}$ and $u(n)$ respectively to distinguish the action of the elements in $L(k,0)$ on $\sigma$-twisted modules and untwisted modules as follows $$Y_{\sigma}(u,z)=\sum_{n\in \Z+\frac{r}{2}}u_{n}z^{-n-1},\ Y(u,z)=\sum_{n\in \Z}u(n)z^{-n-1}.$$
Recall that the top level $U^{i}=\bigoplus_{j=0}^{i}\mathbb{C}v^{i,j}$ of $L(k,i)$ for $0\leq i\leq k$ is an $(i+1)$-dimensional irreducible module for $\C h(0)\oplus \C e(0)\oplus \C f(0)\cong sl_2$. Let
\begin{eqnarray*}
\eta_{i}=\sum_{j=0}^{i}(-1)^{j}v^{i,j},
\end{eqnarray*} then $\eta_{i}$ is the lowest weight vector with weight $-i$ in $(i+1)$-dimensional irreducible module for $\C h^{'}(0)\oplus \C e^{'}(0)\oplus \C f^{'}(0)\cong sl_2$, that is, $f^{'}(0)\eta_{i}=0$ and $h^{'}(0)\eta_{i}=-i\eta_{i}$, and we have:
\begin{lem}\cite{JW}\label{lem:weight}
For the positive integer $k\geq 3$, and $0\leq i\leq k$,
\begin{eqnarray*}L(0)\eta_{i}=\Big(\frac{i(i-k)}{4(k+2)}+\frac{k-1}{16}\Big)\eta_{i}.
\end{eqnarray*}
\end{lem}

By Lemma \ref{lem:weight}, we have
\begin{lem} \begin{eqnarray}\label{lowest2}
L_{\mraff}(0)\eta_{i}=\Big(\frac{i(i-k)}{4(k+2)}+\frac{k}{16}\Big)\eta_{i}.  \label{eq:3.1}
\end{eqnarray}\end{lem}
 We can now construct the $k+1$ inequivalent irreducible $\sigma$-twisted modules of $L(k,0)$.
\begin{thm}\label{thm:construct}
$\overline{L(k,i)}$ for $0\leq i\leq k$ are $k+1$ inequivalent irreducible $\sigma$-twisted modules of $L(k,0)$ generated by $\eta_{i}$.
\end{thm}
\begin{proof}
We just need to notice that $\eta_{i}$ is the lowest weight vector of the $\sigma$-twisted module $\overline{L(k,i)}$, and $h_{0}^{'}\eta_{i}=(h^{'}(0)+\frac{k}{2})\eta_{i}=(-i+\frac{k}{2})\eta_{i}$, this implies that $\overline{L(k,i)}$ for $0\leq i\leq k$ are $k+1$ inequivalent irreducible $\sigma$-twisted modules of $L(k,0)$ generated by $\eta_{i}$.
\end{proof}

We now classify all the irreducible modules of the orbifold vertex operator algebra $L(k,0)^{\sigma}$. Set

\begin{eqnarray}
u^{k,i,1}=\eta_{i}\in L(k,i)(0),\ u^{k,i,2}=(e-f)_{-\frac{1}{2}}\eta_{i}\in L(k,i)(\frac{1}{2}).
 \end{eqnarray}

By applying the results in \cite{DLM3}, we have:

\begin{prop}\label{prop:orbifold1}
For $0\leq i\leq k$, let $\overline{L(k,i)}^{+}$ and $\overline{L(k,i)}^{-}$ be the $L(k,0)^{\sigma}$-modules generated by $u^{k,i,1}$ and $u^{k,i,2}$ respectively. Then $\overline{L(k,i)}^{+}$ and $\overline{L(k,i)}^{-}$ for $0\leq i\leq k$ are irreducible modules of $L(k,0)^{\sigma}$ with the lowest weights
$$L_{\mraff}(0)u^{k,i,1}=\Big(\frac{i(i-k)}{4(k+2)}+\frac{k}{16}\Big)u^{k,i,1}, \ L_{\mraff}(0)u^{k,i,2}=\Big(\frac{i(i-k)}{4(k+2)}+\frac{k+8}{16}\Big)u^{k,i,2}.$$
\end{prop}
Combining Proposition \ref{prop:twisted1} and the results in \cite{DM1}, we have:
\begin{prop}\label{prop:orbifold2} For $0\leq i\leq k$, we have
$$L(k,i)=L(k,i)^{+}\bigoplus L(k,i)^{-},$$
 where $L(k,i)^{+}$ for $i\neq 0$ is an irreducible module of $L(k,0)^{\sigma}$ generated by $\eta_{i}$ with weight $\frac{i(i+2)}{4(k+2)}$, and $L(k,i)^{-}$ for $i\neq 0$ is an irreducible module of $L(k,0)^{\sigma}$ generated by $e^{'}(0)\eta_{i}$ with the same weight $\frac{i(i+2)}{4(k+2)}$. And $L(k,0)^{+}$ is an irreducible module of $L(k,0)^{\sigma}$ generated by $\bf{1}$ with weight $0$, and $L(k,0)^{-}$ is an irreducible module of $L(k,0)^{\sigma}$ generated by $e(-1)\bf{1}$ with weight $1$.
\end{prop}

\begin{rmk}\label{rmk:irreducible} When we consider the basis $\{e,f,h\}$ of $sl_2$ with the automorphism $\tau(e)=-e, \ \tau(f)=-f, \ \tau(h)=h$,  $L(k,i)^{+}$ for $i\neq 0$ can also be viewed as an irreducible module of $L(k,0)^{\tau}$ generated by the lowest weight vector $v^{i,i}$ with weight $\frac{i(i+2)}{4(k+2)}$, and $L(k,i)^{-}$ for $i\neq 0$ can be viewed as an irreducible module of $L(k,0)^{\sigma}$ generated by $e(0)v^{i,i}$ with the same weight $\frac{i(i+2)}{4(k+2)}$.
\end{rmk}

 From the above discussion, we obtain the classification of the irreducible modules for the orbifold vertex operator algebra $L(k,0)^{\sigma}$.

 \begin{thm}\label{thm:orbifold3}
There are $4(k+1)$ inequivalent irreducible modules of $L(k,0)^{\sigma}$ and the lowest weights of these irreducible modules are listed in Proposition \ref{prop:orbifold1} and Proposition \ref{prop:orbifold2}.
\end{thm}

\begin{rmk}\label{rmk:untwistandtwist} We call irreducible modules $L(k,i)^{\pm}$ for $0\leq i\leq k$ untwisted type modules, and  $\overline{L(k,i)}^{\pm}$ for $0\leq i\leq k$ twisted type modules.
\end{rmk}

We now determine the fusion rules for  irreducible modules of $L(k,0)^{\sigma}$. We first prove the following lemma.

\begin{lem}\label{lem:intertwining.} For $0\leq i,j,l\leq k,$ $i+j+l\in 2{\mathbb{Z}}$, $i+j+l\leq 2k$, let $\mathcal{Y}(\cdot,z)$ be an intertwining operator of $L(k,0)$ of type $\left(\begin{array}{c}
L(k,l)\\
L(k,i) \ L(k,j)
\end{array}\right)$.  Define $\widetilde{\mathcal{Y}}(v,z)=\mathcal{Y}(\Delta(h^{''},z)v,z)$ for $v\in L(k,i)$. Then $\widetilde{\mathcal{Y}}(\cdot,z)$ is an intertwining operator of $L(k,0)^{\sigma}$ of type
$\left(\begin{array}{c}
\overline{L(k,l)}\\
L(k,i) \ \overline{L(k,j)}
\end{array}\right)$.
\end{lem}
\begin{proof} The proof is similar to the proof of Proposition 5.4 of \cite{L3}. For simplicity of the notation, we set $\Delta(z)=\Delta(h^{''},z)$, then we have $\Delta(z)\bf{1}=\bf{1}$, $$[L_{\mraff}(-1),\Delta(z)]=-\frac{d}{dz}\Delta(z),$$ and
$$Y_{L(k,i)}(\Delta(z_{2}+z_{0})a,z_{0})\Delta(z_{2})=\Delta(z_{2})Y_{L(k,i)}(a,z_{0})$$ for $a\in L(k,0)^{\sigma}$.
Thus for $a\in L(k,0)^{\sigma}$, $v\in L(k,i)$, we have

\[
z_{0}^{-1}\text{\ensuremath{\delta}}\left(\frac{z_{1}-z_{2}}{z_{0}}\right)Y_{\overline{L(k,l)}}\left(a,z_{1}\right)
\widetilde{\mathcal{Y}}\left(v,z_{2}\right)-z_{0}^{-1}\delta\left(\frac{z_{2}-z_{1}}{-z_{0}}\right)\widetilde{\mathcal{Y}}\left(v,z_{2}\right)Y_{\overline{L(k,j)}}\left(a,z_{1}\right)
\]

\[
=z_{0}^{-1}\text{\ensuremath{\delta}}\left(\frac{z_{1}-z_{2}}{z_{0}}\right)Y_{L(k,l)}\left(\Delta(z_{1})a,z_{1}\right)
\mathcal{Y}\left(\Delta(z_{2})v,z_{2}\right)
\]

\[
-z_{0}^{-1}\delta\left(\frac{z_{2}-z_{1}}{-z_{0}}\right)\mathcal{Y}\left(\Delta(z_{2})v,z_{2}\right)Y_{L(k,j)}\left(\Delta(z_{1})a,z_{1}\right)
\]

\[
=z_{2}^{-1}\text{\ensuremath{\delta}}\left(\frac{z_{1}-z_{0}}{z_{2}}\right)\mathcal{Y}\left(Y_{L(k,i)}(\Delta(z_{1})a,z_{0})\Delta(z_{2})v,z_{2}\right)
\]

\[
=z_{2}^{-1}\text{\ensuremath{\delta}}\left(\frac{z_{1}-z_{0}}{z_{2}}\right)\mathcal{Y}\left(\Delta(z_{2})Y_{L(k,i)}(a,z_{0})v,z_{2}\right)
\]

\[
=z_{2}^{-1}\text{\ensuremath{\delta}}\left(\frac{z_{1}-z_{0}}{z_{2}}\right)\widetilde{\mathcal{Y}}\left(Y_{L(k,i)}(a,z_{0})v,z_{2}\right)
\]
So $\widetilde{\mathcal{Y}}(\cdot,z)$ is an intertwining operator of $L(k,0)^{\sigma}$ of type
$\left(\begin{array}{c}
\overline{L(k,l)}\\
L(k,i) \ \overline{L(k,j)}
\end{array}\right)$.
\end{proof}

 We now determine the contragredient modules of irreducible $L(k,0)^{\sigma}$-modules. First we recall from \cite{DLY2} that the irreducible $K_0$-modules $M^{i,j}$ for $0\leq i\leq k, 0\leq j\leq i-1$ can be realized in the lattice vertex operator algebra $V_{L^{\bot}}$, where $L=\mathbb{Z}\alpha_{1}+\cdots+\mathbb{Z}\alpha_{k}$ with $\langle \alpha_{i}, \alpha_{j} \rangle=2\delta_{ij}$, and $L^{\bot}$ is the dual lattice of $L$. More concretely, the top level of $M^{i,j}$ is a one dimensional space spanned by $v^{i,j}$ and $v^{i,j}$ has the explicit form in $V_{L^{\bot}}$:
\begin{eqnarray}\label{the lowest weight}
v^{0,0}=\mathbf{1}, \ v^{i,0}=\sum\limits_{\tiny{\begin{split}I\subseteq\{1,2,\cdots,k\} \\  |I|=i \ \ \ \ \ \end{split}}}e^{\alpha_{I}/2}, \ v^{i,j}=\sum\limits_{\tiny{\begin{split}I\subseteq\{1,2,\cdots,k\}, \\ |I|=i\ \ \ \ \ \end{split}}}\sum\limits_{\tiny{\begin{split}J\subseteq I \\ |J|=j \ \end{split}}}e^{\alpha_{I-J}/2-\alpha_{J}/2},\label{eq:3.3}
\end{eqnarray}
where $\alpha_{J}=\sum_{i\in J}\alpha_{i}$ for a subset $J$ of $\{1,2\cdots,k\}$, and the vertex operator associated with $e_{\alpha}, \alpha\in L^{\bot}$ is defined on $V_{L^{\bot}}$ by
\begin{eqnarray}
\mathscr{Y}(e_{\alpha},z)=\mbox{exp}\Big(\sum^{\infty}_{n=1}\frac{\alpha(-n)}{n}z^{n}\Big)\mbox{exp}
\Big(-\sum^{\infty}_{n=1}\frac{\alpha(n)}{n}z^{-n}\Big)e_{\alpha}z^{\alpha(0)}.\label{eq:3.4.}
\end{eqnarray}
From \cite[Chapter 8]{FLM} and \cite{DL} (see also \cite{Ab1}), the operator $\mathscr{Y}(\cdot, z)$ produces the intertwining operator for $V_{L}$ of type $\left(\begin{array}{c}
V_{\lambda_{1}+\lambda_{2}+L}\\
V_{\lambda_{1}+L} \ V_{\lambda_{2}+L}
\end{array}\right)$ for $\lambda_{1}, \lambda_{2}\in L^{\bot}$.

 \begin{thm}\label{thm:contragredient}
 For $0\leq i\leq k$. (1) If $i\in 2\mathbb{Z}$, $L(k,i)^{\pm}$ are self-dual. If $i\in 2\mathbb{Z}+1$, then $(L(k,i)^{\pm})^{'}\cong L(k,i)^{\mp}$.
 (2) $(\overline{L(k,i)}^{\pm})'=\overline{L(k,k-i)}^{\pm}$.
\end{thm}
\begin{proof} First we prove (1).
We know that if $M$ is a module of a vertex operator algebra $V$, and $M^{'}$ is the contragredient module of $M$, then $V\subseteq M\boxtimes M^{'}$. Note that $\mathbf{1}\in L(k,0)^{+}$, $v^{i,i}\in L(k,i)^{+}$, and from (\ref{eq:3.3}) we know that \begin{eqnarray}v^{i,i}=\sum\limits_{\tiny{\begin{split}J\subseteq\{1,2,\cdots,k\} \\  |J|=i \ \ \ \ \ \end{split}}}e^{-\alpha_{J}/2}.\end{eqnarray}
Since $\mathbf{1}\in L(k,0)^{+}\subseteq L(k,i)^{+}\boxtimes (L(k,i)^{+})^{'}$, from (\ref{eq:3.4.}), we can deduce that \begin{eqnarray}v^{i,0}=\sum\limits_{\tiny{\begin{split}I\subseteq\{1,2,\cdots,k\} \\  |I|=i \ \ \ \ \ \end{split}}}e^{\alpha_{I}/2} \in (L(k,i)^{+})^{'}.\end{eqnarray}
Note that $$v^{i,0}=\frac{1}{i!}e(0)^{i}v^{i,i}$$ for $i\neq 0$. From Remark \ref{rmk:irreducible}, we know that $v^{i,0}\in L(k,i)^{+}$ if $i\in 2\mathbb{Z}$ and $v^{i,0}\in L(k,i)^{-}$ if $i\in 2\mathbb{Z}+1$. That is,
if $i\in 2\mathbb{Z}$, $L(k,i)^{+}$ is self-dual, and if $i\in 2\mathbb{Z}+1$, $(L(k,i)^{+})^{'}\cong L(k,i)^{-}$. Thus, if $i\in 2\mathbb{Z}$, $L(k,i)^{-}$ is self-dual, and if $i\in 2\mathbb{Z}+1$, $(L(k,i)^{-})^{'}\cong L(k,i)^{+}$.

Next we prove (2), i.e., $(\overline{L(k,i)}^{+})'=\overline{L(k,k-i)}^{+}$. Notice that the top level of the irreducible $L(k,0)^{\sigma}$-module $\overline{L(k,i)}^{+}$ is one-dimensional and spanned by $\eta_{i}$. We denote the top level of the contragredient module $(\overline{L(k,i)}^{+})'$ by $\eta_{i}^{'}$. From the definition of contragredient module (\ref{eq:3.0}), we know that $\eta_{i}$ and $\eta_{i}^{'}$ have the same weight. Thus from (\ref{eq:3.1}) and Proposition \ref{prop:orbifold1}, we know that $\eta_{i}^{'}=\eta_{i}$ or $\eta_{i}^{'}=\eta_{k-i}$. Also from the definition of the contragredient module (\ref{eq:3.0}) and noting that $L_{\mraff}(0)h^{'}=h^{'}$, $L_{\mraff}(1)h^{'}=0$, we have that $$\langle o(h^{'})\eta_{i}^{'}, \eta_{i}\rangle=-\langle \eta_{i}^{'}, o(h^{'})\eta_{i}\rangle,$$ where $o(h^{'})=h^{'}_{\tiny{\mbox{wt}(h)-1}}=h^{'}_{0}.$ Since ${h_{0}^{'}}.\eta_{i}=(-i+\frac{k}{2})\eta_{i}$, it shows that $\eta_{i}^{'}=\eta_{k-i}$, which implies that $(\overline{L(k,i)}^{+})'=\overline{L(k,k-i)}^{+}$. It follows immediately that $(\overline{L(k,i)}^{-})'=\overline{L(k,k-i)}^{-}$.

\end{proof}

For $0\leq i\leq k, \ 0\leq j\leq k,\ 0\leq l\leq k$ such that $i+j+l\in 2\mathbb{Z}$, noticing that $i+j-l\notin 4{\mathbb{Z}}$ is equivalent to $i+j-l+2\in 4\mathbb{Z}$, we define

\[\mbox{sign}(i,j,l)^{+}=\begin{cases}+,\ & \mbox{if} \ i+j-l\in 4{\mathbb{Z}},\cr
-,\ &\mbox{if} \  i+j-l\notin 4{\mathbb{Z}},\end{cases}\]

and

\[\mbox{sign}(i,j,l)^{-}=\begin{cases}-,\ & \mbox{if} \ i+j-l\in 4{\mathbb{Z}},\cr
+,\ &\mbox{if} \  i+j-l\notin 4{\mathbb{Z}}.\end{cases}\]

  The following theorem together with Proposition \ref{fusionsymm.} and  Theorem \ref{thm:contragredient} give all the fusion rules for the ${\Z}_{2}$-orbifold affine vertex operator algebra $L(k,0)^{\sigma}$.

\begin{thm}\label{fusion.aff.} The fusion rules for the ${\mathbb{Z}}_{2}$-orbifold affine vertex operator algebra $L(k,0)^{\sigma}$ are as follows:
\begin{eqnarray}\label{fusion.untwist1..}
L(k,i)^{+}\boxtimes L(k,j)^{\pm}=\sum\limits_{\tiny{\begin{split}|i-j|\leq l\leq i+j \\  i+j+l\in 2\mathbb{Z} \ \ \ \\ i+j+l\leq 2k\ \ \ \end{split}}} L(k,l)^{\mbox{sign}(i,j,l)^{\pm}},
\end{eqnarray}

\begin{eqnarray}\label{fusion.untwist2}
L(k,i)^{-}\boxtimes L(k,j)^{\pm}=\sum\limits_{\tiny{\begin{split}|i-j|\leq l\leq i+j \\  i+j+l\in 2\mathbb{Z} \ \ \ \\ i+j+l\leq 2k\ \ \ \end{split}}}  L(k,l)^{\mbox{sign}(i,j,l)^{\mp}},
\end{eqnarray}

\begin{eqnarray}\label{fusion.twist1}
L(k,i)^{+}\boxtimes \overline{L(k,j)}^{\pm}=\sum\limits_{\tiny{\begin{split}|i-j|\leq l\leq i+j \\  i+j+l\in 2\mathbb{Z} \ \ \ \\ i+j+l\leq 2k\ \ \ \end{split}}} \overline{L(k,l)}^{\mbox{sign}(i,j,l)^{\pm}},
\end{eqnarray}

\begin{eqnarray}\label{fusion.twist2}
L(k,i)^{-}\boxtimes \overline{L(k,j)}^{\pm}=\sum\limits_{\tiny{\begin{split}|i-j|\leq l\leq i+j \\  i+j+l\in 2\mathbb{Z} \ \ \ \\ i+j+l\leq 2k\ \ \ \end{split}}}  \overline{L(k,l)}^{\mbox{sign}(i,j,l)^{\mp}}.
\end{eqnarray}
\end{thm}

\begin{proof} Let $\mathcal{Y}(\cdot,z)$ be an intertwining operator of $L(k,0)$ of type $\left(\begin{array}{c}
L(k,l)\\
L(k,i) \ L(k,j)
\end{array}\right)$. From Lemma \ref{lem:intertwining.}, we know that $\widetilde{\mathcal{Y}}(\cdot,z)$ is an intertwining operator of $L(k,0)^{\sigma}$ of type
$\left(\begin{array}{c}
\overline{L(k,l)}\\
L(k,i) \ \overline{L(k,j)}
\end{array}\right)$, where $\widetilde{\mathcal{Y}}(v,z)=\mathcal{Y}(\Delta(h^{''},z)v,z)$ for $v\in L(k,i)$. Thus we have
\begin{eqnarray}\label{fusion}
\widetilde{\mathcal{Y}}(\eta_{i},z)=\mathcal{Y}(\Delta(h^{''},z)\eta_{i},z)=z^{-\frac{i}{4}}\mathcal{Y}(\eta_{i},z),
\end{eqnarray}
where $\eta_{i}=\sum_{j=0}^{i}(-1)^{j}v^{i,j}$ is the lowest weight vector of the $\sigma$-twisted module $\overline{L(k,i)}$.
From  (\ref{lowest1}), we know that $\eta_{i}$ has the weight $\frac{i(i+2)}{4(k+2)}$ in $L(k,i)$ for $0\leq i\leq k$. For simplicity, we denote $a_{i}=\frac{i(i+2)}{4(k+2)}$. From (\ref{lowest2}), we know that $\eta_{i}$ has the weight $\frac{i(i-k)}{4(k+2)}+\frac{k}{16}$ in $\overline{L(k,i)}$ for $0\leq i\leq k$. We denote $\tilde{a}_{i}=\frac{i(i-k)}{4(k+2)}+\frac{k}{16}$. From Lemma \ref{affine}, we know that the fusion rule of the affine vertex operator algebra $L(k,0)$ is $$L(k,i)\boxtimes_{L(k,0)} L(k,j)=\sum\limits_{l} L(k,l),$$
where $|i-j|\leq l\leq i+j, \ i+j+l\in 2\mathbb{Z},\ i+j+l\leq 2k.$ From (\ref{fusion}), we have
\begin{eqnarray}\label{fusionidentity}\widetilde{\mathcal{Y}}(\eta_{i},z)\eta_{j}=\mathcal{Y}(\Delta(h^{''},z)\eta_{i},z)\eta_{j}
=z^{-\frac{i}{4}}\mathcal{Y}(\eta_{i},z)\eta_{j},
\end{eqnarray}
which implies that the fact that $\widetilde{\mathcal{Y}}$ is the intertwining operator of $L(k,0)^{\sigma}$ of type
$\left(\begin{array}{c}
{\overline{L(k,l)}}^{+}\\
L(k,i)^{+} \ {\overline{L(k,j)}^{+}}
\end{array}\right)$ is equivalent to
\begin{eqnarray}\label{fusion2}
a_{i}+a_{j}-a_{l}-\tilde{a}_{i}-\tilde{a}_{j}+\tilde{a}_{l}+\frac{i}{4}\in \mathbb{Z},
\end{eqnarray}
that is, $i+j-l\in 4\mathbb{Z}$.
And the fact that $\widetilde{\mathcal{Y}}$ is the intertwining operator of $L(k,0)^{\sigma}$ of type
$\left(\begin{array}{c}
{\overline{L(k,l)}}^{-}\\
L(k,i)^{+} \ {\overline{L(k,j)}^{+}}
\end{array}\right)$ is equivalent to
\begin{eqnarray}\label{fusion2}
a_{i}+a_{j}-a_{l}-\tilde{a}_{i}-\tilde{a}_{j}+\tilde{a}_{l}+\frac{i}{4}+\frac{1}{2}\in \mathbb{Z},
\end{eqnarray}
that is, $i+j-l+2\in 4\mathbb{Z}$. Since $i+j+l\in 2\mathbb{Z}$, it follows that $i+j-l+2\in 4\mathbb{Z}$ is equivalent to $i+j-l\notin 4\mathbb{Z}$. Thus from the definition of the symbol $\mbox{sign}(i,j,l)$, we obtain (\ref{fusion.twist1}) and (\ref{fusion.twist2}). Note that $\sigma(h^{''})=h^{''}$, thus $\eta_{i}$ and $\Delta(h^{''},z)\eta_{i}$ are in the same irreducible untwisted module of $L(k,0)^{\sigma}$, then by (\ref{fusionidentity}),  (\ref{fusion.twist1}) and (\ref{fusion.twist2}), we obtain (\ref{fusion.untwist1..}) and (\ref{fusion.untwist2}).
\end{proof}

\section{Quantum dimensions for irreducible $K_{0}^{\sigma}$-modules
}\label{Sect:quantum dimension}\def\theequation{4.\arabic{equation}}

In this section, we first recall some results on the quantum dimensions of irreducible $g$-twisted modules and irreducible $V^{G}$-modules for $G$ being a finite automorphism group of the vertex operator algebra $V$ following \cite{DRX}. Then we determine the quantum dimensions for irreducible modules of the orbifold vertex operator algebra $K_{0}^{\sigma}$.



 We now recall some notions about quantum dimensions. Let $V$ be a vertex operator algebra, $g$ an automorphism of $V$ with order $T$ and $M=\oplus_{n\in\frac{1}{T}\mathbb{Z}_{+}}M_{\lambda+n}$ a $g$-twisted $V$-module.

\begin{definition} For an homogeneous element $v\in V$, a trace function associated to $v$ is defined as follows:

\[
Z_{M}(v,q)=\mbox{tr}_{M}o(v)q^{L\left(0\right)-c/24}=q^{\lambda-c/24}\sum_{n\in\frac{1}{T}\mathbb{Z}_{+}}\mbox{tr}M_{\lambda+n}o(v)q^{n},
\]
where $o(v)=v(\mbox{wt}v-1)$ is the degree zero operator of $v$, $c$ is the central charge of the vertex operator algebra $V$
and $\lambda$ is the conformal weight of $M$. \end{definition}

It is proved \cite{Z,DLM4} that $Z_{M}(v,q)$ converges to a
holomorphic function in the domain $|q|<1$ if $V$ is $C_{2}$-cofinite. We denote the holomorphic
function $Z_{M}(v,q)$ by $Z_{M}\left(v,\tau\right)$. Here and
below, $\tau$ is in the upper half plane $\mathbb{H}$ and $q=e^{2\pi i\tau}$. Note that if $v=1$ is the vacuum vector, then $Z_{M}(1,q)$ is the formal character of $M$ and we denote $Z_{M}(1,q)$ and $Z_{M}(1,\tau)$ by $\chi_{M}(q)$ and $\chi_{M}(\tau)$ respectively for simplicity. $\chi_{M}(q)$ is called the character of $M$.







Let $V$ be a rational, $C_2$-cofinite, and selfdual vertex operator algebra of CFT type, and $G$ a finite automorphism group of $V$. Let $g\in G$ and $M$ a $g$-twisted $V$-module. The quantum dimension of $M$ over $V$ is defined to be $$\mbox{qdim}_{V}M=\lim_{y\to 0}\frac{\chi_{M}\left(iy\right)}{\chi_{V}\left(iy\right)},$$
 where $y$ is real and positive\cite{DRX}.



From \cite{M2} and \cite{CM}, we have
\begin{thm}\label{orbifold} If $V$ is a regular, selfdual vertex operator algebra of CFT type, and $G$ is solvable, then $V^{G}$ is a regular, selfdual vertex operator algebra of CFT type.
\end{thm}

From now on, we assume $V$ is a rational, $C_{2}$-cofinite vertex
operator algebra of CFT type with $V\cong V'$. Let $M^{0}\cong V,\, M^{1},\,\cdots,\, M^{d}$
denote all inequivalent irreducible $V$-modules. Moreover, we assume
the conformal weights $\lambda_{i}$ of $M^{i}$ are positive for
all $i>0.$ From Theorem \ref{orbifold}, the orbifold parafermion vertex operator algebra $K_0^{\sigma}$
satisfies all the assumptions.

The following result shows that the quantum dimensions are multiplicative under tensor product \cite{DJX} .

\begin{prop}\label{quantum-product} Let $V$ and $M_i$ for $0\leq i\leq d$ be as above. Then
\[
\mbox{qdim}_{V}\left(M^{i}\boxtimes M^{j}\right)=\mbox{qdim}_{V}M^{i}\cdot \mbox{qdim}_{V}M^{j}
\]
for $i,\, j=0,\cdots,\, d.$
\end{prop}
Recalling from \cite{DLY2}, let $L=\mathbb{Z}\alpha_{1}+\cdots+\mathbb{Z}\alpha_{k}$ with $\langle \alpha_{i}, \alpha_{j} \rangle=2\delta_{ij}$ and let $\gamma=\alpha_{1}+\cdots+\alpha_{k}$, then $\langle \gamma, \gamma \rangle=2k$.
$V_{\mathbb{Z}\gamma}$ is the vertex operator algebra associated with a rank one lattice $\mathbb{Z}\gamma$ and as a $V_{\mathbb{Z}\gamma}\otimes K_0$-module (note that $K_0=M^{0,0}$), $L(k,i)$ has a decomposition:
\begin{eqnarray}
L(k,i)=\bigoplus_{j=0}^{k-1}V_{\mathbb{Z}\gamma+(i-2j)\gamma/2k}\otimes M^{i,j} \ \ \ \mbox{for} \ 0\leq i\leq k,\label{eq:4.0.}
\end{eqnarray}
where $V_{\mathbb{Z}\gamma+(i-2j)\gamma/2k}$ are the irreducible modules of the lattice vertex operator algebra $V_{\mathbb{Z}\gamma}$. Since every irreducible $V_{\mathbb{Z}\gamma}$-module is a simple current, we have
\begin{eqnarray}
\mbox{qdim}_{V_{\mathbb{Z}\gamma}}V_{\mathbb{Z}\gamma+(i-2j)\gamma/2k}=1.\label{eq:4.1.}
\end{eqnarray}

We get the following result on the quantum dimension of the orbifold parafermion vertex operator algebra $K_0^{\sigma}$.

\begin{thm}\label{quantum-dimension} The quantum dimensions for all irreducible $K_{0}^{\sigma}$-modules are
\begin{eqnarray} \mbox{qdim}_{K_{0}^{\sigma}}W(k,i)^{j}=\sqrt{k}\frac{\sin\frac{\pi(i+1)}{k+2}}{\sin\frac{\pi}{k+2}} \ \mbox{for} \ 0\leq i\leq k, i\neq \frac{k}{2} \ \mbox{if} \ k \ \mbox{is even}, j=1,2,\label{eq:4.3.}\end{eqnarray}
\begin{eqnarray} \mbox{qdim}_{K_{0}^{\sigma}}W(k,\frac{k}{2})^{j}=\mbox{qdim}_{K_{0}^{\sigma}}\widetilde{W(k,\frac{k}{2})}^{j}=\frac{\sqrt{k}}{2}\frac{\sin\frac{\pi(\frac{k}{2}+1)}{k+2}}{\sin\frac{\pi}{k+2}} \ \mbox{for} \ j=1,2,\label{eq:4.4.}\end{eqnarray}
\begin{eqnarray} \mbox{qdim}_{K_{0}^{\sigma}}(M^{i,j})^{s}=\frac{\sin\frac{\pi(i+1)}{k+2}}{\sin\frac{\pi}{k+2}}, \ \ s=0,1 \label{eq:4.5.}\end{eqnarray}
for $(M^{i,j})^{s}$ being the untwisted $K_{0}^{\sigma}$-module of type $I$.
\begin{eqnarray} \mbox{qdim}_{K_{0}^{\sigma}}M^{i,j}=2\frac{\sin\frac{\pi(i+1)}{k+2}}{\sin\frac{\pi}{k+2}} \label{eq:4.6.}\end{eqnarray}
for $M^{i,j}$ being the untwisted $K_{0}^{\sigma}$-module of type $II$.
\end{thm}
\begin{proof} Since the quantum dimensions of irreducible modules $L(k,i)$ of affine vertex operator algebra $L(k,0)$ are
\begin{eqnarray*}
\mbox{qdim}_{L(k,0)}L(k,i)=\frac{\sin\frac{\pi(i+1)}{k+2}}{\sin\frac{\pi}{k+2}}
\end{eqnarray*}
for $0\leq i\leq k$. From  Proposition 4.1 of \cite{DJJJY}, we know that
\begin{eqnarray*}
\mbox{qdim}_{L(k,0)}\overline{L(k,i)}=\frac{\sin\frac{\pi(i+1)}{k+2}}{\sin\frac{\pi}{k+2}}.
\end{eqnarray*}
Since from \cite{DW3}, \begin{eqnarray}
\mbox{qdim}_{K_0}M^{i,j}=\frac{\sin\frac{\pi(i+1)}{k+2}}{\sin\frac{\pi}{k+2}},
\end{eqnarray}
 together with (\ref{eq:4.1.}), we have
 \begin{eqnarray*}
\mbox{qdim}_{V_{\mathbb{Z}\gamma}\otimes K_0}V_{\mathbb{Z}\gamma+(i-2j)\gamma/2k}\otimes M^{i,j}=\frac{\sin\frac{\pi(i+1)}{k+2}}{\sin\frac{\pi}{k+2}}.
\end{eqnarray*}
Thus, from (\ref{eq:4.0.}), we have
\begin{eqnarray*}
\mbox{qdim}_{V_{\mathbb{Z}\gamma}\otimes K_0}L(k,i)=k\frac{\sin\frac{\pi(i+1)}{k+2}}{\sin\frac{\pi}{k+2}}.
\end{eqnarray*}
From  Proposition 4.1 of \cite{DJJJY}, we have \begin{eqnarray}
\mbox{qdim}_{V_{\mathbb{Z}\gamma}\otimes K_0}\overline{L(k,i)}=k\frac{\sin\frac{\pi(i+1)}{k+2}}{\sin\frac{\pi}{k+2}}.\label{eq:4.8.}
\end{eqnarray}
Recall from \cite{JW} that all the irreducible twisted modules $W(k,i)$ of $K_0$ come from $\overline{L(k,i)}$ for $0\leq i\leq k$, or more precisely, for the fixed $i\neq\frac{k}{2}$, $W(k,i)$ is the only irreducible twisted module of $K_0$, and if $i=\frac{k}{2}$, there are two irreducible twisted modules $W(k,\frac{k}{2})$ and $\widetilde{W(k,\frac{k}{2})}$ of $K_0$. Note that if $i\neq\frac{k}{2}$, as the twisted module of the vertex operator algebra $V_{\mathbb{Z}\gamma}\otimes K_0$, $\overline{L(k,i)}$ has a decomposition:
\begin{eqnarray}
\overline{L(k,i)}=V_{\mathbb{Z}\gamma}^{T_{a_{i}}}\otimes W(k,i),\label{eq:4.9.}
\end{eqnarray}
where $a_{i}=1$ or $2$ depending on $i$, $V_{\mathbb{Z}\gamma}^{T_{a_{i}}}\in\{V_{\mathbb{Z}\gamma}^{T_{1}},\ V_{\mathbb{Z}\gamma}^{T_{2}}\}$, and $V_{\mathbb{Z}\gamma}^{T_{1}},\ V_{\mathbb{Z}\gamma}^{T_{2}}$ are the irreducible twisted $V_{\mathbb{Z}\gamma}$-modules \cite{D}. For $i=\frac{k}{2}$,
\begin{eqnarray}
\overline{L(k,\frac{k}{2})}=V_{\mathbb{Z}\gamma}^{T_{a_{\frac{k}{2}}}}\otimes W(k,\frac{k}{2})+V_{\mathbb{Z}\gamma}^{T^{'}_{a_{\frac{k}{2}}}}\otimes \widetilde{W(k,\frac{k}{2})},\label{eq:4.10.}
\end{eqnarray}
as a $V_{\mathbb{Z}\gamma}\otimes K_0$-twisted module, where $V_{\mathbb{Z}\gamma}^{T_{a_{\frac{k}{2}}}}, V_{\mathbb{Z}\gamma}^{T^{'}_{a_{\frac{k}{2}}}}\in \{V_{\mathbb{Z}\gamma}^{T_{1}},\ V_{\mathbb{Z}\gamma}^{T_{2}}\}$.
From \cite{DRX}, we know that $\mbox{qdim}_{V_{\mathbb{Z}\gamma}}V_{\mathbb{Z}\gamma}^{T_{i}}=\sqrt{k}$ for $i=1,2$. Together with (\ref{eq:4.8.}), (\ref{eq:4.9.}), (\ref{eq:4.10.}), we have
\begin{eqnarray*}
\mbox{qdim}_{K_0}W(k,i)=\sqrt{k}\frac{\sin\frac{\pi(i+1)}{k+2}}{\sin\frac{\pi}{k+2}}
\end{eqnarray*}
for $i\neq \frac{k}{2}$.
\begin{eqnarray*}
\mbox{qdim}_{K_0}W(k,\frac{k}{2})=\mbox{qdim}_{K_0}\widetilde{W(k,\frac{k}{2})}=\frac{\sqrt{k}}{2}\frac{\sin\frac{\pi(i+1)}{k+2}}{\sin\frac{\pi}{k+2}}.
\end{eqnarray*}
From the Theorem 4.4 of \cite{DRX}, we have
\begin{eqnarray*}
\mbox{qdim}_{K_{0}^{\sigma}}W(k,i)^{j}=\sqrt{k}\frac{\sin\frac{\pi(i+1)}{k+2}}{\sin\frac{\pi}{k+2}}
\end{eqnarray*}
for $i\neq \frac{k}{2}$, $j=1,2$, which proves (\ref{eq:4.3.}). Furthermore,
\begin{eqnarray*}
\mbox{qdim}_{K_{0}^{\sigma}}W(k,\frac{k}{2})^{j}=\mbox{qdim}_{K_{0}^{\sigma}}\widetilde{W(k,\frac{k}{2})}^{j}=\frac{\sqrt{k}}{2}\frac{\sin\frac{\pi(\frac{k}{2}+1)}{k+2}}{\sin\frac{\pi}{k+2}}
\end{eqnarray*}
for $j=1,2$, proving (\ref{eq:4.4.}). Since
\begin{eqnarray*}
\mbox{qdim}_{K_0}M^{i,j}=\frac{\sin\frac{\pi(i+1)}{k+2}}{\sin\frac{\pi}{k+2}},
\end{eqnarray*}
from  Corollary 4.5 of \cite{DRX}, we have
\begin{eqnarray*}
\mbox{qdim}_{K_{0}^{\sigma}}M^{i,j}=2\frac{\sin\frac{\pi(i+1)}{k+2}}{\sin\frac{\pi}{k+2}},
\end{eqnarray*}
for $M^{i,j}$ being the untwisted $K_{0}^{\sigma}$-module of type $II$, which proves (\ref{eq:4.6.}).
Finally we have
\begin{eqnarray*}
\mbox{qdim}_{K_{0}^{\sigma}}(M^{i,j})^{s}=\frac{\sin\frac{\pi(i+1)}{k+2}}{\sin\frac{\pi}{k+2}},\ \ s=0,1
\end{eqnarray*}
for $(M^{i,j})^{s}$ being the untwisted $K_{0}^{\sigma}$-module of type $I$. We obtain (\ref{eq:4.5.}).

\end{proof}
\section{Fusion rules for $\Z_{2}$-orbifold of the parafemion vertex operator algebra $K_{0}$
}\label{Sect: fusion product}\def\theequation{5.\arabic{equation}}
\setcounter{equation}{0}

In this section, we give the fusion rules for $K_{0}^{\sigma}$. To emphasize the action of the automorphism $\sigma$, we denote twisted type modules $W(k,i)^{1}$ by $W(k,i)^{+}$ and $W(k,i)^{2}$ by $W(k,i)^{-}$, and we denote $\widetilde{W(k,\frac{k}{2})}^{1}$ by $\widetilde{W(k,\frac{k}{2})}^{+}$ and  $\widetilde{W(k,\frac{k}{2})}^{2}$ by $\widetilde{W(k,\frac{k}{2})}^{-}$. We denote untwisted modules $(M^{i,j})^{0}$ of type $I$ by $(M^{i,j})^{+}$ and  $(M^{i,j})^{1}$ by $(M^{i,j})^{-}$. For the irreducible $K_{0}^{\sigma}$-modules $W^{1}$ and $W^{2}$, we use $W^{1}\boxtimes W^{2}$ to denote the fusion product $W^{1}\boxtimes_{K_{0}^{\sigma}}W^{2}$
for simplicity in this section.

 We first give the fusion rules for all the untwisted type modules.

\begin{thm}\label{para-fusion-untwist} The fusion rules for the irreducible untwisted type modules of the ${\mathbb{Z}}_{2}$-orbifold parafermion vertex operator algebra $K_{0}^{\sigma}$ are as follows:

(1) If $k\in 2\mathbb{Z}+1$, i.e., $k=2n+1$ for $n\geq 1$, we have
\begin{eqnarray}
(M^{k,0})^{+}\boxtimes (M^{i,j})^{\pm}=(M^{i,j})^{\pm},\label{eq:5.1}
\end{eqnarray}
where $(i,j)=(i,\frac{i}{2}), \ i=2,4,6,\cdots,2n,$ or $(i,j)=(2n+1,0)$.

\begin{eqnarray}
(M^{k,0})^{-}\boxtimes (M^{i,j})^{\pm}=(M^{i,j})^{\mp},\label{eq:5.1'}
\end{eqnarray}
where $(i,j)=(i,\frac{i}{2}), \ i=2,4,6,\cdots,2n,$ or $(i,j)=(2n+1,0)$.

\begin{eqnarray}
(M^{i,\frac{i}{2}})^{+}\boxtimes (M^{j,\frac{j}{2}})^{\pm}=\sum\limits_{\tiny{\begin{split}|i-j|\leq l\leq i+j \\  i+j+l\in 2\mathbb{Z} \ \ \ \\ i+j+l\leq 2k\ \ \ \end{split}}} (M^{l,\overline{(\frac{l}{2})}})^{\mbox{sign}(i,j,l)^{{\pm}}},\label{eq:5.2}
\end{eqnarray}

\begin{eqnarray}
(M^{i,\frac{i}{2}})^{-}\boxtimes (M^{j,\frac{j}{2}})^{\pm}=\sum\limits_{\tiny{\begin{split}|i-j|\leq l\leq i+j \\  i+j+l\in 2\mathbb{Z} \ \ \ \\ i+j+l\leq 2k\ \ \ \end{split}}} (M^{l,\overline{(\frac{l}{2})}})^{\mbox{sign}(i,j,l)^{{\mp}}},\label{eq:5.2'}
\end{eqnarray}
where $\overline{a}$ means the residue of the integer $a$ modulo $k$. The following is the same, which we will not point out again.

(2) If $k\in 2\mathbb{Z}$, i.e., $k=2n$ for $n\geq 2$, we have

\begin{eqnarray}
(M^{k,0})^{+}\boxtimes (M^{i,j})^{\pm}=(M^{i,j})^{\pm},\label{eq:5.3}
\end{eqnarray}

\begin{eqnarray}
(M^{k,0})^{-}\boxtimes (M^{i,j})^{\pm}=(M^{i,j})^{\mp},\label{eq:5.3'}
\end{eqnarray}
where $(i,j)=(i,\frac{i}{2}), \ i=2,4,6,\cdots,2n,$ $(i,j)=(n,0)$ or $(i,j)=(2n,0)$.

\begin{eqnarray}
(M^{i,\frac{i}{2}})^{+}\boxtimes (M^{j,\frac{j}{2}})^{\pm}=\sum\limits_{\tiny{\begin{split}|i-j|\leq l\leq i+j \\  i+j+l\in 2\mathbb{Z} \ \ \ \\ i+j+l\leq 2k\ \ \ \end{split}}} (M^{l,\overline{(\frac{l}{2})}})^{\mbox{sign}(i,j,l)^{{\pm}}},\label{eq:5.4}
\end{eqnarray}

\begin{eqnarray}
(M^{i,\frac{i}{2}})^{-}\boxtimes (M^{j,\frac{j}{2}})^{\pm}=\sum\limits_{\tiny{\begin{split}|i-j|\leq l\leq i+j \\  i+j+l\in 2\mathbb{Z} \ \ \ \\ i+j+l\leq 2k\ \ \ \end{split}}} (M^{l,\overline{(\frac{l}{2})}})^{\mbox{sign}(i,j,l)^{{\mp}}},\label{eq:5.4'}
\end{eqnarray}


\begin{eqnarray}
(M^{i,\frac{i}{2}})^{+}\boxtimes (M^{\frac{k}{2},0})^{\pm}=\sum\limits_{\tiny{\begin{split}|\frac{k}{2}-i|\leq l<\frac{k}{2} \\  i+j+l\in 2\mathbb{Z} \ \ \ \\ i+j+l\leq 2k\ \ \ \end{split}}} M^{l,\overline{(\frac{2l-k}{4}})}+(M^{\frac{k}{2},0})^{\pm},\label{eq:5.5}
\end{eqnarray}

\begin{eqnarray}
(M^{i,\frac{i}{2}})^{-}\boxtimes (M^{\frac{k}{2},0})^{\pm}=\sum\limits_{\tiny{\begin{split}|\frac{k}{2}-i|\leq l<\frac{k}{2} \\  i+j+l\in 2\mathbb{Z} \ \ \ \\ i+j+l\leq 2k\ \ \ \end{split}}} M^{l,\overline{(\frac{2l-k}{4}})}+(M^{\frac{k}{2},0})^{\mp},\label{eq:5.5'}
\end{eqnarray}


\begin{eqnarray}
(M^{\frac{k}{2},0})^{+}\boxtimes (M^{\frac{k}{2},0})^{\pm}=\sum\limits_{\tiny{\begin{split}0\leq l\leq k \\  k+l\in 2\mathbb{Z}  \\ l\leq k \ \ \end{split}}} (M^{k-l,\overline{(\frac{k-l}{2}})})^{\pm},\label{eq:5.6}
\end{eqnarray}

\begin{eqnarray}
(M^{\frac{k}{2},0})^{-}\boxtimes (M^{\frac{k}{2},0})^{\pm}=\sum\limits_{\tiny{\begin{split}0\leq l\leq k \\  k+l\in 2\mathbb{Z}  \\ l\leq k \ \ \end{split}}}(M^{k-l,\overline{(\frac{k-l}{2}})})^{\mp}.\label{eq:5.6'}
\end{eqnarray}

(3) If $k\in \mathbb{Z}$ and $k\geq 3$, we have
\begin{eqnarray}\begin{split}
&(M^{i,i^{'}})^{+}\boxtimes M^{j,j^{'}}=(M^{i,i^{'}})^{-}\boxtimes M^{j,j^{'}}\\
&=\sum\limits_{l} \Big((M^{l,\overline{\frac{1}{2}(2i^{'}-i+2j^{'}-j+l)}})^{+}
+(M^{l,\overline{\frac{1}{2}(2i^{'}-i+2j^{'}-j+l)}})^{-}\Big)
+\sum\limits_{l^{'}}  M^{l^{'},\overline{\frac{1}{2}(2i^{'}-i+2j^{'}-j+l^{'})}},
\label{eq:5.7}\end{split}
\end{eqnarray}
where $(M^{i,i^{'}})^{\pm}$ are untwisted modules of type $I$, $M^{j,j^{'}}$ are untwisted modules of type $II$, and $|i-j|\leq l\leq i+j, \ i+j+l\in 2\mathbb{Z},\ i+j+l\leq 2k$ such that
$(M^{l,\overline{\frac{1}{2}(2i^{'}-i+2j^{'}-j+l)}})^{\pm}$ are irreducible untwisted modules of type $I$. $|i-j|\leq l^{'}\leq i+j, \ i+j+l^{'}\in 2\mathbb{Z},\ i+j+l^{'}\leq 2k$ such that $M^{l^{'},\overline{\frac{1}{2}(2i^{'}-i+2j^{'}-j+l^{'})}}$ are irreducible untwisted modules of type $II$, Moreover, with fixed $i,i^{'},j,j^{'}$, $(M^{l,\overline{\frac{1}{2}(2i^{'}-i+2j^{'}-j+l)}})^{\pm}$ for $|i-j|\leq l\leq i+j, \ i+j+l\in 2\mathbb{Z},\ i+j+l\leq 2k$ are inequivalent irreducible modules. $M^{l^{'},\overline{\frac{1}{2}(2i^{'}-i+2j^{'}-j+l^{'})}}$ for $|i-j|\leq l^{'}\leq i+j, \ i+j+l^{'}\in 2\mathbb{Z},\ i+j+l^{'}\leq 2k$ are inequivalent irreducible modules.

(4) If $k\in \mathbb{Z}$ and $k\geq 3$, we have
\begin{eqnarray}
\begin{split}
M^{i,i^{'}}\boxtimes M^{j,j^{'}}&=\sum\limits_{l} \Big((M^{l,\overline{\frac{1}{2}(2i^{'}-i+2j^{'}-j+l)}})^{+}+(M^{l,\overline{\frac{1}{2}(2i^{'}-i+2j^{'}-j+l)}})^{-}\\
&+(M^{l,\overline{\frac{1}{2}(2i^{'}-i+2(j-j^{'})-j+l)}})^{+}+(M^{l,\overline{\frac{1}{2}(2i^{'}-i+2(j-j^{'})-j+l)}})^{-}\Big)\\
&+\sum\limits_{l^{'}}  \Big(M^{l^{'},\overline{\frac{1}{2}(2i^{'}-i+2j^{'}-j+l^{'})}}+M^{l^{'},\overline{\frac{1}{2}(2i^{'}-i+2(j-j^{'})-j+l^{'})}}\Big),\label{eq:5.8}
\end{split}
\end{eqnarray}
where $M^{i,i^{'}}, M^{j,j^{'}}$ are untwisted modules of type $II$, and $|i-j|\leq l\leq i+j, \ i+j+l\in 2\mathbb{Z},\ i+j+l\leq 2k$ such that
$(M^{l,\overline{\frac{1}{2}(2i^{'}-i+2j^{'}-j+l)}})^{\pm}$, $(M^{l,\overline{\frac{1}{2}(2i^{'}-i+2(j-j^{'})-j+l)}})^{\pm}$ are irreducible untwisted modules of type $I$. $|i-j|\leq l^{'}\leq i+j, \ i+j+l^{'}\in 2\mathbb{Z},\ i+j+l^{'}\leq 2k$ such that $M^{l^{'},\overline{\frac{1}{2}(2i^{'}-i+2j^{'}-j+l^{'})}}$, $M^{l^{'},\overline{\frac{1}{2}(2i^{'}-i+2(j-j^{'})-j+l^{'})}}$ are irreducible untwisted modules of type $II$, Moreover, with fixed $i,i^{'},j,j^{'}$, $(M^{l,\overline{\frac{1}{2}(2i^{'}-i+2j^{'}-j+l)}})^{\pm}$ and $(M^{l,\overline{\frac{1}{2}(2i^{'}-i+2(j-j^{'})-j+l)}})^{\pm}$ for $|i-j|\leq l\leq i+j, \ i+j+l\in 2\mathbb{Z},\ i+j+l\leq 2k$ are inequivalent irreducible $K_0^{\sigma}$-modules. $M^{l^{'},\overline{\frac{1}{2}(2i^{'}-i+2j^{'}-j+l^{'})}}$ and $M^{l^{'},\overline{\frac{1}{2}(2i^{'}-i+2(j-j^{'})-j+l^{'})}}$ for $|i-j|\leq l^{'}\leq i+j, \ i+j+l^{'}\in 2\mathbb{Z},\ i+j+l^{'}\leq 2k$ are inequivalent irreducible $K_0^{\sigma}$-modules.
\end{thm}

\begin{proof} Note that $(M^{k,0})^{+}=K_{0}^{+}$. Let $(M^{i,j}, Y_{M^{i,j}})$ for $1\leq i\leq k,\ 0\leq j\leq i-1$ be irreducible $K_0$-modules, then the operator $Y_{M^{i,j}}$ gives the nonzero intertwining operators for $K_0$ of type $\left(\begin{array}{c}
M^{i,j}\\
K_{0\ }M^{i,j}
\end{array}\right)$. Then by Lemma \ref{intertwining}, $Y_{M^{i,j}}(a,z)v$ is nonzero for any nonzero vectors $a\in K_0$, $v\in M^{i,j}$. Since $\sigma Y_{M^{i,j}}(a,z)\sigma^{-1}=Y_{M^{i,j}}(\sigma(a),z)$ for $a\in K_{0}$, $Y_{M^{i,j}}$ gives the nonzero intertwining operators for $K_{0}^{+}$ of type $\left(\begin{array}{c}
(M^{i,j})^{\pm}\\
K_{0}^{+} \ (M^{i,j})^{\pm}
\end{array}\right)$ and $\left(\begin{array}{c}
(M^{i,j})^{\mp}\\
K_{0}^{-} \ (M^{i,j})^{\pm}
\end{array}\right)$. This implies (\ref{eq:5.1}), (\ref{eq:5.1'}), (\ref{eq:5.3}), (\ref{eq:5.3'}).

For (\ref{eq:5.2}), (\ref{eq:5.2'}), (\ref{eq:5.4}), (\ref{eq:5.4'}), from \cite{DW3}, we know that
\begin{eqnarray}
M^{i,i^{'}}\boxtimes_{K_0} M^{j,j^{'}}=\sum\limits_{\tiny{\begin{split}|i-j|\leq l\leq i+j \\  i+j+l\in 2\mathbb{Z} \ \ \ \\ i+j+l\leq 2k\ \ \ \end{split}}} M^{l,\overline{\frac{1}{2}(2i^{'}-i+2j^{'}-j+l)}}.\label{eq:5.8.}
\end{eqnarray}
Thus we have
\begin{eqnarray*}
(M^{i,\frac{i}{2}})^{+}\boxtimes (M^{j,\frac{j}{2}})^{+}\subseteq \sum\limits_{\tiny{\begin{split}|i-j|\leq l\leq i+j \\  i+j+l\in 2\mathbb{Z} \ \ \ \\ i+j+l\leq 2k\ \ \ \end{split}}} M^{l,\overline{(\frac{l}{2})}}.
\end{eqnarray*}
Since $(M^{i,\frac{i}{2}})^{+}\subseteq L(k,i)^{+}$, and $L(k,i)$ has a decomposition (\ref{eq:4.0.}):
\begin{eqnarray*}
L(k,i)=\bigoplus_{j=0}^{k-1}V_{\mathbb{Z}\gamma+(i-2j)\gamma/2k}\otimes M^{i,j} \ \ \ \mbox{for} \ 0\leq i\leq k,
\end{eqnarray*}
we have $$V_{\mathbb{Z}\gamma}^{+}\otimes (M^{i,\frac{i}{2}})^{+}\subseteq L(k,i)^{+},\ V_{\mathbb{Z}\gamma}^{+}\otimes (M^{j,\frac{j}{2}})^{+}\subseteq L(k,j)^{+}.$$
Moreover, from Theorem \ref{fusion.aff.}, we know
\begin{eqnarray}\label{fusion.untwist1}
L(k,i)^{+}\boxtimes L(k,j)^{+}=\sum\limits_{\tiny{\begin{split}|i-j|\leq l\leq i+j \\  i+j+l\in 2\mathbb{Z} \ \ \ \\ i+j+l\leq 2k\ \ \ \end{split}}} L(k,l)^{\mbox{sign}(i,j,l)^{+}}.
\end{eqnarray}
Together with the facts that $V_{\mathbb{Z}\gamma}^{+}\boxtimes V_{\mathbb{Z}\gamma}^{+}=V_{\mathbb{Z}\gamma}^{+}$,
$ \mbox{qdim}_{K_{0}^{\sigma}}(M^{i,\frac{i}{2}})^{+}=\frac{\sin\frac{\pi(i+1)}{k+2}}{\sin\frac{\pi}{k+2}},  \
 $
and
$\mbox{qdim}_{K_{0}^{\sigma}}\Big((M^{i,\frac{i}{2}})^{+}\boxtimes(M^{j,\frac{j}{2}})^{+}\Big)=
\mbox{qdim}_{K_{0}^{\sigma}}(M^{i,\frac{i}{2}})^{+}\cdot\mbox{qdim}_{K_{0}^{\sigma}}(M^{j,\frac{j}{2}})^{+},
$ we can deduce that
\begin{eqnarray*}
(M^{i,\frac{i}{2}})^{+}\boxtimes (M^{j,\frac{j}{2}})^{+}=\sum\limits_{\tiny{\begin{split}|i-j|\leq l\leq i+j \\  i+j+l\in 2\mathbb{Z} \ \ \ \\ i+j+l\leq 2k\ \ \ \end{split}}} (M^{l,\overline{(\frac{l}{2})}})^{\mbox{sign}(i,j,l)^{{+}}}.
\end{eqnarray*}
Since
\begin{eqnarray*}
K_0^{-}\boxtimes (M^{i,\frac{i}{2}})^{+}\boxtimes (M^{j,\frac{j}{2}})^{+}=(M^{i,\frac{i}{2}})^{-}\boxtimes (M^{j,\frac{j}{2}})^{+}=(M^{i,\frac{i}{2}})^{+}\boxtimes (M^{j,\frac{j}{2}})^{-},
\end{eqnarray*}
and \begin{eqnarray*}
K_0^{-}\boxtimes (M^{i,\frac{i}{2}})^{+}\boxtimes (M^{j,\frac{j}{2}})^{-}=(M^{i,\frac{i}{2}})^{-}\boxtimes (M^{j,\frac{j}{2}})^{-}=(M^{i,\frac{i}{2}})^{+}\boxtimes (M^{j,\frac{j}{2}})^{+},
\end{eqnarray*}
together with (\ref{eq:5.8.}), we obtain (\ref{eq:5.2}), (\ref{eq:5.2'}), (\ref{eq:5.4}), (\ref{eq:5.4'}).

For (\ref{eq:5.5}) and (\ref{eq:5.5'}), from the fusion rule (\ref{eq:5.8.}) of irreducible $K_0$-modules, we have
\begin{eqnarray*}
(M^{i,\frac{i}{2}})^{+}\boxtimes (M^{\frac{k}{2},0})^{+}\subseteq \sum\limits_{\tiny{\begin{split}|\frac{k}{2}-i|\leq l\leq \frac{k}{2}+i, l\neq\frac{k}{2} \\  i+\frac{k}{2}+l\in 2\mathbb{Z} \ \ \ \\ i+\frac{k}{2}+l\leq 2k\ \ \ \end{split}}} M^{l,\overline{(\frac{2l-k}{4})}}+M^{\frac{k}{2},0}.
\end{eqnarray*}
Note that $M^{l,\overline{(\frac{2l-k}{4})}}$ for $|\frac{k}{2}-i|\leq l\leq \frac{k}{2}+i, l\neq\frac{k}{2}$ are irreducible modules of $K_{0}^{\sigma}$,
i.e., they are the untwisted modules of type $II$, and
we have $M^{l,\overline{(\frac{2l-k}{4})}}\cong M^{k-l,\overline{(\frac{k-2l}{4})}}$. Note that $M^{\frac{k}{2},0}=(M^{\frac{k}{2},0})^{+}+(M^{\frac{k}{2},0})^{-}$ as $K_0^{\sigma}$-module. From Theorem \ref{quantum-dimension}, we have
\begin{eqnarray*} \mbox{qdim}_{K_{0}^{\sigma}}(M^{i,\frac{i}{2}})^{+}=\frac{\sin\frac{\pi(i+1)}{k+2}}{\sin\frac{\pi}{k+2}},  \
 \mbox{qdim}_{K_{0}^{\sigma}}(M^{\frac{k}{2},0})^{+}=\frac{\sin\frac{\pi(\frac{k}{2}+1)}{k+2}}{\sin\frac{\pi}{k+2}},  \
\mbox{qdim}_{K_{0}^{\sigma}}M^{l,\overline{(\frac{2l-k}{4})}}=2\frac{\sin\frac{\pi(l+1)}{k+2}}{\sin\frac{\pi}{k+2}}.\end{eqnarray*}
By using
\begin{eqnarray*}\mbox{qdim}_{K_{0}^{\sigma}}\Big((M^{i,\frac{i}{2}})^{+}\boxtimes(M^{\frac{k}{2},0})^{+}\Big)=
\mbox{qdim}_{K_{0}^{\sigma}}(M^{i,\frac{i}{2}})^{+}\cdot\mbox{qdim}_{K_{0}^{\sigma}}(M^{\frac{k}{2},0})^{+},
\end{eqnarray*}
and noting that if $i\leq\frac{k}{2}$, then $l_{\tiny\mbox{min}}=\frac{k}{2}-i, \ l_{\tiny\mbox{max}}=\frac{k}{2}+i$, we have
$$\sum\limits_{\frac{k}{2}-i\leq l<\frac{k}{2}}\frac{\sin\frac{\pi(l+1)}{k+2}}{\sin\frac{\pi}{k+2}}=\sum\limits_{\frac{k}{2}< l\leq i+\frac{k}{2}}\frac{\sin\frac{\pi(l+1)}{k+2}}{\sin\frac{\pi}{k+2}}.$$
If $i>\frac{k}{2}$, then $l_{\tiny\mbox{min}}=i-\frac{k}{2}, \ l_{\tiny\mbox{max}}=2k-i-\frac{k}{2}=\frac{3k}{2}-i$. Thus
$$\sum\limits_{i-\frac{k}{2}\leq l<\frac{k}{2}}\frac{\sin\frac{\pi(l+1)}{k+2}}{\sin\frac{\pi}{k+2}}=\sum\limits_{\frac{k}{2}< l\leq \frac{3k}{2}-i}\frac{\sin\frac{\pi(l+1)}{k+2}}{\sin\frac{\pi}{k+2}}.$$
So we have
\begin{eqnarray*}
(M^{i,\frac{i}{2}})^{+}\boxtimes (M^{\frac{k}{2},0})^{+}=\sum\limits_{\tiny{\begin{split}|\frac{k}{2}-i|\leq l<\frac{k}{2} \\  i+\frac{k}{2}+l\in 2\mathbb{Z} \ \ \ \\ i+\frac{k}{2}+l\leq 2k\ \ \ \end{split}}} M^{l,\overline{(\frac{2l-k}{4})}}+(M^{\frac{k}{2},0})^{\epsilon},
\end{eqnarray*}
where $\epsilon=+$ or $-$. We now prove that $\epsilon=+$. Since we have mentioned in Section 3 that the irreducible modules $M^{i,j}$ for $1\leq i\leq k,\ 0\leq j\leq i-1$ can be realized in the lattice vertex operator algebra $V_{L^{\bot}}$, and
\begin{eqnarray*}
v^{i,\frac{i}{2}}=\sum\limits_{\tiny{\begin{split}I\subseteq\{1,2,\cdots,k\}, \\ |I|=i\ \ \ \ \ \end{split}}}\sum\limits_{\tiny{\begin{split}J\subseteq I, \\ |J|=\frac{i}{2} \ \end{split}}}e^{\alpha_{I-J}/2-\alpha_{J}/2}\in (M^{i,\frac{i}{2}})^{+},
\end{eqnarray*}
 in this case, we notice that $|I-J|=|J|$ and
 \begin{eqnarray*}
v^{\frac{k}{2},0}=\sum\limits_{\tiny{\begin{split}I\subseteq\{1,2,\cdots,k\} \\ |I|=\frac{k}{2}\ \ \ \ \ \end{split}}} e^{\alpha_{I}/2}\in (M^{\frac{k}{2},0})^{+}.
\end{eqnarray*}
From (\ref{eq:3.4.}), we can deduce that $v^{\frac{k}{2},0}$ can be obtained from $\mathscr{Y}^{\circ}(v^{i,\frac{i}{2}},z)v^{\frac{k}{2},0}$,
where $\mathscr{Y}^{\circ}$ is the nonzero intertwining operator for $V_{L}$ of type $\left(\begin{array}{c}
V_{\lambda_{1}+\lambda_{2}+L}\\
V_{\lambda_{1}+L} \ V_{\lambda_{2}+L}
\end{array}\right)$ for $\lambda_{1}, \lambda_{2}\in L^{\bot}$. Since $v^{\frac{k}{2},0}\in(M^{\frac{k}{2},0})^{+}$, this shows that $\epsilon=+$. Similar to the discussion in the end of the proof of (\ref{eq:5.2}), (\ref{eq:5.2'}), we obtain (\ref{eq:5.5}) and (\ref{eq:5.5'}).

For (\ref{eq:5.6}) and (\ref{eq:5.6'}), from the fusion rule (\ref{eq:5.8.}) of irreducible $K_0$-modules, we have
\begin{eqnarray*}
(M^{\frac{k}{2},0})^{+}\boxtimes (M^{\frac{k}{2},0})^{+}\subseteq \sum\limits_{\tiny{\begin{split}0 \leq l\leq k \\  k+l\in 2\mathbb{Z}  \\ k+l\leq 2k  \end{split}}} M^{l,\overline{(\frac{l-k}{2})}}.
\end{eqnarray*}
Note that
$M^{l,\overline{(\frac{l-k}{2})}}\cong M^{k-l,\overline{(\frac{k-l}{2})}}$ as $K_0$-modules, and
$M^{k-l,\overline{(\frac{k-l}{2})}}=(M^{k-l,\overline{(\frac{k-l}{2})}})^{+}\oplus (M^{k-l,\overline{(\frac{k-l}{2})}})^{-}$ as a $K_0^{\sigma}$-module.
From Theorem \ref{quantum-dimension}, we have
\begin{eqnarray*}
 \mbox{qdim}_{K_{0}^{\sigma}}(M^{\frac{k}{2},0})^{+}=\frac{\sin\frac{\pi(\frac{k}{2}+1)}{k+2}}{\sin\frac{\pi}{k+2}},  \
\mbox{qdim}_{K_{0}^{\sigma}}(M^{k-l,\overline{(\frac{k-l}{2})}})^{+}=\mbox{qdim}_{K_{0}^{\sigma}}(M^{k-l,\overline{(\frac{k-l}{2})}})^{-}
=\frac{\sin\frac{\pi(k-l+1)}{k+2}}{\sin\frac{\pi}{k+2}}.\end{eqnarray*}
By using
\begin{eqnarray*}\mbox{qdim}_{K_{0}^{\sigma}}\Big((M^{\frac{k}{2},0})^{+}\boxtimes(M^{\frac{k}{2},0})^{+}\Big)=
\mbox{qdim}_{K_{0}^{\sigma}}(M^{\frac{k}{2},0})^{+}\cdot\mbox{qdim}_{K_{0}^{\sigma}}(M^{\frac{k}{2},0})^{+},
\end{eqnarray*}
we can deduce that
\begin{eqnarray*}
(M^{\frac{k}{2},0})^{+}\boxtimes (M^{\frac{k}{2},0})^{+}=\sum\limits_{\tiny{\begin{split}0 \leq l\leq k \\  k+l\in 2\mathbb{Z}  \\ k+l\leq 2k  \end{split}}} (M^{l,\overline{(\frac{l-k}{2})}})^{\epsilon_{l}},
\end{eqnarray*}
where $\epsilon_{l}=+$ or $-$. We now prove that $\epsilon_{l}=+$.
Since \begin{eqnarray*}
v^{\frac{k}{2},0}=\sum\limits_{\tiny{\begin{split}I\subseteq\{1,2,\cdots,k\} \\ |I|=\frac{k}{2}\ \ \ \ \ \end{split}}} e^{\alpha_{I}/2}\in (M^{\frac{k}{2},0})^{+},
\end{eqnarray*}
and $M^{\frac{k}{2},0}\cong M^{\frac{k}{2},\frac{k}{2}}$ as $K_0$-module, we have
\begin{eqnarray*}
v^{\frac{k}{2},\frac{k}{2}}=\sum\limits_{\tiny{\begin{split}J\subseteq\{1,2,\cdots,k\} \\ |J|=\frac{k}{2}\ \ \ \ \ \end{split}}} e^{-\alpha_{J}/2}\in (M^{\frac{k}{2},0})^{+}.
\end{eqnarray*}
Then from (\ref{eq:3.4.}), we know that
\begin{eqnarray*}
v^{k-l,\frac{k-l}{2}}=\sum\limits_{\tiny{\begin{split}I\subseteq\{1,2,\cdots,k\}, \\ |I|=k-l\ \ \ \ \ \end{split}}}\sum\limits_{\tiny{\begin{split}J\subseteq I, \\ |J|=\frac{k-l}{2} \ \end{split}}}e^{\alpha_{I-J}/2-\alpha_{J}/2}\in (M^{l,\overline{(\frac{l-k}{2})}})^{+}
\end{eqnarray*}
can be obtained from $\mathscr{Y}^{\circ}(v^{\frac{k}{2},0},z)v^{\frac{k}{2},\frac{k}{2}}$,
where $\mathscr{Y}^{\circ}$ is the nonzero intertwining operator for $V_{L}$ of type $\left(\begin{array}{c}
V_{\lambda_{1}+\lambda_{2}+L}\\
V_{\lambda_{1}+L} \ V_{\lambda_{2}+L}
\end{array}\right)$ for $\lambda_{1}, \lambda_{2}\in L^{\bot}$. This shows that $\epsilon_{l}=+$. Similar to the discussion in the end of the proof of (\ref{eq:5.2}), (\ref{eq:5.2'}), we get (\ref{eq:5.6}) and (\ref{eq:5.6'}).

For (\ref{eq:5.7}), notice that
\begin{eqnarray*}
\begin{split}
M^{j,j^{'}}\boxtimes (M^{i,i^{'}})^{+}&=K_{0}^{+}\boxtimes M^{j,j^{'}}\boxtimes (M^{i,i^{'}})^{+}\\
&=K_{0}^{-}\boxtimes M^{j,j^{'}}\boxtimes (M^{i,i^{'}})^{+}\\
&=M^{j,j^{'}}\boxtimes K_{0}^{-}\boxtimes (M^{i,i^{'}})^{+}\\
&=M^{j,j^{'}}\boxtimes (M^{i,i^{'}})^{-},
\end{split}
\end{eqnarray*}
where $M^{j,j^{'}}$ are the untwisted modules of type $II$, and $(M^{i,i^{'}})^{+}$ are the untwisted modules of type $I$. From the fusion rule (\ref{eq:5.8.}) of irreducible $K_0$-modules, we have $I\left(\begin{array}{c}(M^{l,\overline{\frac{1}{2}(2i^{'}-i+2j^{'}-j+l)}})\\
(M^{i,i^{'}})^{+} \ \ M^{j,j^{'}}\end{array}\right)\neq 0$ for $|i-j|\leq l\leq i+j,\   i+j+l\in 2\mathbb{Z}, \  i+j+l\leq 2k$. From Theorem \ref{quantum-dimension}, we have
\begin{eqnarray*}
 \mbox{qdim}_{K_{0}^{\sigma}}(M^{i,i^{'}})^{+}=\frac{\sin\frac{\pi(i+1)}{k+2}}{\sin\frac{\pi}{k+2}},  \
\mbox{qdim}_{K_{0}^{\sigma}}M^{j,j^{'}}
=2\frac{\sin\frac{\pi(j+1)}{k+2}}{\sin\frac{\pi}{k+2}}.\end{eqnarray*}
By using
\begin{eqnarray*}\mbox{qdim}_{K_{0}^{\sigma}}\Big((M^{i,i^{'}})^{+}\boxtimes M^{j,j^{'}}\Big)=
\mbox{qdim}_{K_{0}^{\sigma}}(M^{i,i^{'}})^{+}\cdot\mbox{qdim}_{K_{0}^{\sigma}}M^{j,j^{'}},
\end{eqnarray*}
we can deduce that (\ref{eq:5.7}) hold. The second assertion follows from Theorem 4.2 of \cite{DW3} immediately.

For (\ref{eq:5.8}), from the fusion rule (\ref{eq:5.8.}) of irreducible $K_0$-modules, we have
\begin{eqnarray*}
M^{i,i^{'}}\boxtimes_{K_0} M^{j,j^{'}}=\sum\limits_{\tiny{\begin{split}|i-j|\leq l\leq i+j \\  i+j+l\in 2\mathbb{Z} \ \ \ \\ i+j+l\leq 2k\ \ \ \end{split}}} M^{l,\overline{\frac{1}{2}(2i^{'}-i+2j^{'}-j+l)}},
\end{eqnarray*}
where $M^{i,i^{'}}, M^{j,j^{'}}$ are the untwisted $K_0^{\sigma}$-modules of type $II$. From \cite{JW}, we know that $M^{i,i^{'}}\cong M^{i,i-i^{'}}$, $M^{j,j^{'}}\cong M^{j,j-j^{'}}$ as $K_0^{\sigma}$-module. Thus
\begin{eqnarray*}
M^{i,i^{'}}\boxtimes_{K_0^{\sigma}} M^{j,j^{'}}=M^{i,i^{'}}\boxtimes_{K_0^{\sigma}} M^{j,j-j^{'}}=M^{i,i-i^{'}}\boxtimes_{K_0^{\sigma}} M^{j,j-j^{'}}.
\end{eqnarray*}
Note that
\begin{eqnarray*}
M^{i,i^{'}}\boxtimes_{K_0} M^{j,j-j^{'}}=\sum\limits_{\tiny{\begin{split}|i-j|\leq l\leq i+j \\  i+j+l\in 2\mathbb{Z} \ \ \ \\ i+j+l\leq 2k\ \ \ \end{split}}} M^{l,\overline{\frac{1}{2}(2i^{'}-i+2(j-j^{'})-j+l)}},
\end{eqnarray*}
\begin{eqnarray*}
M^{i,i-i^{'}}\boxtimes_{K_0} M^{j,j-j^{'}}=\sum\limits_{\tiny{\begin{split}|i-j|\leq l\leq i+j \\  i+j+l\in 2\mathbb{Z} \ \ \ \\ i+j+l\leq 2k\ \ \ \end{split}}} M^{l,\overline{\frac{1}{2}(2(i-i^{'})-i+2(j-j^{'})-j+l)}}.
\end{eqnarray*}
We claim that $$M^{l,\overline{\frac{1}{2}(2i^{'}-i+2j^{'}-j+l)}}\cong M^{l,\overline{\frac{1}{2}(2(i-i^{'})-i+2(j-j^{'})-j+l)}}$$ for $|i-j|\leq l\leq i+j, \  i+j+l\in 2\mathbb{Z}, \ i+j+l\leq 2k$. If we can prove the claim, then we have
\begin{eqnarray*}
\sum\limits_{\tiny{\begin{split}|i-j|\leq l\leq i+j \\  i+j+l\in 2\mathbb{Z} \ \ \ \\ i+j+l\leq 2k\ \ \ \end{split}}} M^{l,\overline{\frac{1}{2}(2i^{'}-i+2j^{'}-j+l)}}+\sum\limits_{\tiny{\begin{split}|i-j|\leq l\leq i+j \\  i+j+l\in 2\mathbb{Z} \ \ \ \\ i+j+l\leq 2k\ \ \ \end{split}}} M^{l,\overline{\frac{1}{2}(2i^{'}-i+2(j-j^{'})-j+l)}}\subseteq M^{i,i^{'}}\boxtimes_{K_0^{\sigma}} M^{j,j^{'}}.
\end{eqnarray*}
Moreover, from Theorem \ref{quantum-dimension}, we have
\begin{eqnarray*}
 \mbox{qdim}_{K_{0}^{\sigma}}(M^{i,i^{'}})=2\frac{\sin\frac{\pi(i+1)}{k+2}}{\sin\frac{\pi}{k+2}},  \
\mbox{qdim}_{K_{0}^{\sigma}}M^{j,j^{'}}
=2\frac{\sin\frac{\pi(j+1)}{k+2}}{\sin\frac{\pi}{k+2}}.\end{eqnarray*}
Then (\ref{eq:5.8}) follows from
\begin{eqnarray*}\mbox{qdim}_{K_{0}^{\sigma}}\Big(M^{i,i^{'}}\boxtimes M^{j,j^{'}}\Big)=
\mbox{qdim}_{K_{0}^{\sigma}}M^{i,i^{'}}\cdot\mbox{qdim}_{K_{0}^{\sigma}}M^{j,j^{'}}.
\end{eqnarray*}
 The second assertion follows from Theorem 4.2 of \cite{DW3} immediately. We now prove the claim, i.e., $$M^{l,\overline{\frac{1}{2}(2i^{'}-i+2j^{'}-j+l)}}\cong M^{l,\overline{\frac{1}{2}(2(i-i^{'})-i+2(j-j^{'})-j+l)}}$$ for $|i-j|\leq l\leq i+j, \  i+j+l\in 2\mathbb{Z}, \ i+j+l\leq 2k.$
If $M^{l,\overline{\frac{1}{2}(2i^{'}-i+2j^{'}-j+l)}}$ is the untwisted $K_0^{\sigma}$-modules of type $II$, then from \cite{JW}, we have
\begin{eqnarray*}
M^{l,\overline{\frac{1}{2}(2i^{'}-i+2j^{'}-j+l)}}\cong M^{l,\overline{l-\frac{1}{2}(2i^{'}-i+2j^{'}-j+l)}}= M^{l,\overline{\frac{1}{2}(2(i-i^{'})-i+2(j-j^{'})-j+l)}}.
\end{eqnarray*}
If $M^{l,\overline{\frac{1}{2}(2i^{'}-i+2j^{'}-j+l)}}$ is the untwisted $K_0^{\sigma}$-modules of type $I$, we divide the proof of the claim into three cases:

(i) If $(l, \overline{\frac{1}{2}(2i^{'}-i+2j^{'}-j+l)})=(l,\overline{\frac{l}{2}})$, then $\overline{\frac{1}{2}(2i^{'}-i+2j^{'}-j+l)}=\overline{\frac{l}{2}}$, i.e., $\overline{2i^{'}-i}=\overline{j-2j^{'}}$, thus

\begin{eqnarray*}
M^{l,\overline{\frac{1}{2}(2(i-i^{'})-i+2(j-j^{'})-j+l)}}=M^{l,\overline{(\frac{l}{2})}}=M^{l,\overline{\frac{1}{2}(2i^{'}-i+2j^{'}-j+l)}}.
\end{eqnarray*}

(ii) If $(l, \overline{\frac{1}{2}(2i^{'}-i+2j^{'}-j+l)})=(k,\bar{0})$, then $\overline{\frac{1}{2}(2i^{'}-i+2j^{'}-j+k)}=\bar{0}$, i.e., $\overline{2i^{'}-i}=\overline{j-2j^{'}-k}$, thus

\begin{eqnarray*}
M^{l,\overline{\frac{1}{2}(2(i-i^{'})-i+2(j-j^{'})-j+l)}}=M^{k,\bar{k}}=M^{k,\bar{0}}=M^{l,\overline{\frac{1}{2}(2i^{'}-i+2j^{'}-j+l)}}.
\end{eqnarray*}

(iii) If $(l, \overline{\frac{1}{2}(2i^{'}-i+2j^{'}-j+l)})=(\frac{k}{2},\bar{0})$, then $\overline{\frac{1}{2}(2i^{'}-i+2j^{'}-j+\frac{k}{2})}=\bar{0}$, i.e., $\overline{2i^{'}-i}=\overline{j-2j^{'}-\frac{k}{2}}$, thus

\begin{eqnarray*}
M^{l,\overline{\frac{1}{2}(2(i-i^{'})-i+2(j-j^{'})-j+l)}}=M^{\frac{k}{2},\overline{(\frac{k}{2})}}=M^{\frac{k}{2},\bar{0}}=M^{l,\overline{\frac{1}{2}(2i^{'}-i+2j^{'}-j+l)}}.
\end{eqnarray*}
Thus we proved the claim.
\end{proof}

We now give the fusion products between untwisted type modules and twisted type modules.
\begin{thm}\label{para-fusion-twist} The fusion rules for the irreducible untwisted type modules and twisted type modules of the ${\mathbb{Z}}_{2}$-orbifold parafermion vertex operator algebra $K_{0}^{\sigma}$ are as follows:

(1) If $k\in 2\mathbb{Z}+1$, $0\leq j\leq \frac{k-1}{2}$, we have

\begin{eqnarray}
(M^{k,0})^{+}\boxtimes W(k,j)^{\pm}=W(k,j)^{\pm},\label{eq:5.9}
\end{eqnarray}

\begin{eqnarray}
(M^{k,0})^{-}\boxtimes W(k,j)^{\pm}=W(k,j)^{\mp}.\label{eq:5.9'}
\end{eqnarray}

If $k\in 2\mathbb{Z}$, $0\leq j\leq \frac{k}{2}$, we have
\begin{eqnarray}
(M^{k,0})^{+}\boxtimes W(k,j)^{\pm}=W(k,j)^{\pm},\label{eq:5.9.}
\end{eqnarray}

\begin{eqnarray}
(M^{k,0})^{-}\boxtimes W(k,j)^{\pm}=W(k,j)^{\mp},\label{eq:5.9.'}
\end{eqnarray}

\begin{eqnarray}
(M^{k,0})^{+}\boxtimes \widetilde{W(k,\frac{k}{2})}^{\pm}=\widetilde{W(k,\frac{k}{2})}^{\pm},\label{eq:.5.9.}
\end{eqnarray}

\begin{eqnarray}
(M^{k,0})^{-}\boxtimes \widetilde{W(k,\frac{k}{2})}^{\pm}=\widetilde{W(k,\frac{k}{2})}^{\mp}.\label{eq:.5.9.'}
\end{eqnarray}

(2) For $(M^{i,\frac{i}{2}})^{+}$ being the untwisted module of type $I$, we have the following results:\\

If $k\in 2\mathbb{Z}+1$, $0\leq j\leq \frac{k-1}{2}$, we have
\begin{eqnarray}
(M^{i,\frac{i}{2}})^{+}\boxtimes W(k,j)^{\pm}=\sum\limits_{\tiny{\begin{split}|i-j|\leq l\leq i+j \\  i+j+l\in 2\mathbb{Z} \ \ \ \\ i+j+l\leq 2k\ \ \ \end{split}}} W(k,l)^{\mbox{sign}(i,j,l)^{\pm}},\label{eq:5.10}
\end{eqnarray}

\begin{eqnarray}
(M^{i,\frac{i}{2}})^{-}\boxtimes W(k,j)^{\pm}=\sum\limits_{\tiny{\begin{split}|i-j|\leq l\leq i+j \\  i+j+l\in 2\mathbb{Z} \ \ \ \\ i+j+l\leq 2k\ \ \ \end{split}}} W(k,l)^{\mbox{sign}(i,j,l)^{\mp}}.\label{eq:5.10'}
\end{eqnarray}

If $k\in 4\Z+2$, $i+j\in 2\mathbb{Z}$, or $k\in 4\Z$, $i+j\in 2\mathbb{Z}+1$, we have
\begin{eqnarray}
(M^{i,\frac{i}{2}})^{+}\boxtimes W(k,j)^{\pm}=\sum\limits_{\tiny{\begin{split}|i-j|\leq l\leq i+j \\  i+j+l\in 2\mathbb{Z} \ \ \ \\ i+j+l\leq 2k\ \ \ \end{split}}}W(k,l)^{\mbox{sign}(i,j,l)^{\pm}},\label{eq:5.16}
\end{eqnarray}
and
\begin{eqnarray}
(M^{i,\frac{i}{2}})^{-}\boxtimes W(k,j)^{\pm}=\sum\limits_{\tiny{\begin{split}|i-j|\leq l\leq i+j \\  i+j+l\in 2\mathbb{Z} \ \ \ \\ i+j+l\leq 2k\ \ \ \end{split}}}W(k,l)^{\mbox{sign}(i,j,l)^{\mp}}.\label{eq:5.16'}
\end{eqnarray}

 If $k\in 4\Z+2$, $i+j\in 2\mathbb{Z}+1$, or $k\in 4\Z$, $i+j\in2\Z$. And $i+j<\frac{k}{2}$ or $|i-j|>\frac{k}{2}$, $j\neq\frac{k}{2}$, we have
\begin{eqnarray}
(M^{i,\frac{i}{2}})^{+}\boxtimes W(k,j)^{\pm}=\sum\limits_{\tiny{\begin{split}|i-j|\leq l\leq i+j \\  i+j+l\in 2\mathbb{Z} \ \ \ \\ i+j+l\leq 2k\ \ \ \end{split}}}W(k,l)^{\mbox{sign}(i,j,l)^{\pm}},\label{eq:5.17}
\end{eqnarray}

\begin{eqnarray}
(M^{i,\frac{i}{2}})^{-}\boxtimes W(k,j)^{\pm}=\sum\limits_{\tiny{\begin{split}|i-j|\leq l\leq i+j \\  i+j+l\in 2\mathbb{Z} \ \ \ \\ i+j+l\leq 2k\ \ \ \end{split}}}W(k,l)^{\mbox{sign}(i,j,l)^{\mp}}.\label{eq:5.17'}
\end{eqnarray}

If $k\in 4\Z+2$, $i+j\in 2\mathbb{Z}+1$, or $k\in 4\Z$, $i+j\in2\Z$. And $i+j\geq\frac{k}{2}\geq |i-j|$, $j\neq\frac{k}{2}$, we have
\begin{eqnarray}
(M^{i,\frac{i}{2}})^{+}\boxtimes W(k,j)^{\pm}=\sum\limits_{\tiny{\begin{split}|i-j|\leq l\leq i+j \\  i+j+l\in 2\mathbb{Z} \ \ \ \\ i+j+l\leq 2k\ \ \ \end{split}}}W(k,l)^{\mbox{sign}(i,j,l)^{\pm}}+\widetilde{W(k,\frac{k}{2})}^{\mbox{sign}(i,j,l)^{\mp}},\label{eq:5.18}
\end{eqnarray}

\begin{eqnarray}
(M^{i,\frac{i}{2}})^{-}\boxtimes W(k,j)^{\pm}=\sum\limits_{\tiny{\begin{split}|i-j|\leq l\leq i+j \\  i+j+l\in 2\mathbb{Z} \ \ \ \\ i+j+l\leq 2k\ \ \ \end{split}}}W(k,l)^{\mbox{sign}(i,j,l)^{\mp}}+\widetilde{W(k,\frac{k}{2})}^{\mbox{sign}(i,j,l)^{\pm}}.\label{eq:5.18'}
\end{eqnarray}

 If $k\in 2\Z$, $i\in 4\mathbb{Z}+2$, we have
\begin{eqnarray}
(M^{i,\frac{i}{2}})^{+}\boxtimes W(k,\frac{k}{2})^{\pm}=\sum\limits_{\tiny{\begin{split}|i-\frac{k}{2}|\leq l<\frac{k}{2} \\ i+\frac{k}{2}+l\in 2\mathbb{Z} \  \\ i+l\leq \frac{3k}{2}\ \ \ \end{split}}}W(k,l)^{\mbox{sign}(i,\frac{k}{2},l)^{\pm}}+\widetilde{W(k,\frac{k}{2})}^{\pm},\label{eq:5.19}
\end{eqnarray}

\begin{eqnarray}
(M^{i,\frac{i}{2}})^{-}\boxtimes W(k,\frac{k}{2})^{\pm}=\sum\limits_{\tiny{\begin{split}|i-\frac{k}{2}|\leq l<\frac{k}{2} \\ i+\frac{k}{2}+l\in 2\mathbb{Z} \  \\ i+l\leq \frac{3k}{2}\ \ \ \end{split}}}W(k,l)^{\mbox{sign}(i,\frac{k}{2},l)^{\mp}}+\widetilde{W(k,\frac{k}{2})}^{\mp},\label{eq:5.19'}
\end{eqnarray}

\begin{eqnarray}
(M^{i,\frac{i}{2}})^{+}\boxtimes \widetilde{W(k,\frac{k}{2})}^{\pm}=\sum\limits_{\tiny{\begin{split}|i-\frac{k}{2}|\leq l<\frac{k}{2} \\ i+\frac{k}{2}+l\in 2\mathbb{Z} \  \\ i+l\leq \frac{3k}{2}\ \ \ \end{split}}}W(k,l)^{\mbox{sign}(i,\frac{k}{2},l)^{\pm}}+W(k,\frac{k}{2})^{\pm},\label{eq:5.20}
\end{eqnarray}

\begin{eqnarray}
(M^{i,\frac{i}{2}})^{-}\boxtimes \widetilde{W(k,\frac{k}{2})}^{\pm}=\sum\limits_{\tiny{\begin{split}|i-\frac{k}{2}|\leq l<\frac{k}{2} \\ i+\frac{k}{2}+l\in 2\mathbb{Z} \  \\ i+l\leq \frac{3k}{2}\ \ \ \end{split}}}W(k,l)^{\mbox{sign}(i,\frac{k}{2},l)^{\mp}}+W(k,\frac{k}{2})^{\mp}.\label{eq:5.20'}
\end{eqnarray}

 If $k\in 2\Z$, $i\in 4\mathbb{Z}$, we have
\begin{eqnarray}
(M^{i,\frac{i}{2}})^{+}\boxtimes W(k,\frac{k}{2})^{\pm}=\sum\limits_{\tiny{\begin{split}|i-\frac{k}{2}|\leq l<\frac{k}{2} \\ i+\frac{k}{2}+l\in 2\mathbb{Z} \  \\ i+l\leq \frac{3k}{2}\ \ \ \end{split}}}W(k,l)^{\mbox{sign}(i,\frac{k}{2},l)^{\pm}}+W(k,\frac{k}{2})^{\pm},\label{eq:5.21}
\end{eqnarray}

\begin{eqnarray}
(M^{i,\frac{i}{2}})^{-}\boxtimes W(k,\frac{k}{2})^{\pm}=\sum\limits_{\tiny{\begin{split}|i-\frac{k}{2}|\leq l<\frac{k}{2} \\ i+\frac{k}{2}+l\in 2\mathbb{Z} \  \\ i+l\leq \frac{3k}{2}\ \ \ \end{split}}}W(k,l)^{\mbox{sign}(i,\frac{k}{2},l)^{\mp}}+W(k,\frac{k}{2})^{\mp},\label{eq:5.21'}
\end{eqnarray}

\begin{eqnarray}
(M^{i,\frac{i}{2}})^{+}\boxtimes \widetilde{W(k,\frac{k}{2})}^{\pm}=\sum\limits_{\tiny{\begin{split}|i-\frac{k}{2}|\leq l<\frac{k}{2} \\ i+\frac{k}{2}+l\in 2\mathbb{Z} \  \\ i+l\leq \frac{3k}{2}\ \ \ \end{split}}}W(k,l)^{\mbox{sign}(i,\frac{k}{2},l)^{\pm}}+\widetilde{W(k,\frac{k}{2})}^{\pm},\label{eq:5.22}
\end{eqnarray}

\begin{eqnarray}
(M^{i,\frac{i}{2}})^{-}\boxtimes \widetilde{W(k,\frac{k}{2})}^{\pm}=\sum\limits_{\tiny{\begin{split}|i-\frac{k}{2}|\leq l<\frac{k}{2} \\ i+\frac{k}{2}+l\in 2\mathbb{Z} \  \\ i+l\leq \frac{3k}{2}\ \ \ \end{split}}}W(k,l)^{\mbox{sign}(i,\frac{k}{2},l)^{\mp}}+\widetilde{W(k,\frac{k}{2})}^{\mp}.\label{eq:5.22'}
\end{eqnarray}

(3) For $M^{i,i^{'}}$ being the untwisted modules of type $II$, we have the following results:\\

If $k\in 2\Z+1$, we have
\begin{eqnarray}\begin{split}
&M^{i,i^{'}}\boxtimes W(k,j)^{+}=M^{i,i^{'}}\boxtimes W(k,j)^{-}\\
&=\sum\limits_{\tiny{\begin{split}|i-j|\leq l\leq i+j \\  i+j+l\in 2\mathbb{Z} \ \ \ \\ i+j+l\leq 2k\ \ \ \end{split}}} \Big(W(k,l)^{+}+W(k,l)^{-}\Big),
\label{eq:5.11}\end{split}
\end{eqnarray}

If $k\in4\Z+2$, $i+j\in 2\mathbb{Z}+1$, or $k\in4\Z$, $i+j\in 2\mathbb{Z}$. And $j\neq\frac{k}{2}$, $i+j<\frac{k}{2}$ or $|i-j|>\frac{k}{2}$, we have
\begin{eqnarray}\begin{split}
&M^{i,i^{'}}\boxtimes W(k,j)^{+}=M^{i,i^{'}}\boxtimes W(k,j)^{-}\\
&=\sum\limits_{\tiny{\begin{split}|i-j|\leq l\leq i+j \\  i+j+l\in 2\mathbb{Z} \ \ \ \\ i+j+l\leq 2k\ \ \ \end{split}}}  \Big(W(k,l)^{+}+W(k,l)^{-}\Big).
\label{eq:5.12}\end{split}
\end{eqnarray}

If $k\in4\Z+2$, $i+j\in 2\mathbb{Z}+1$, or $k\in4\Z$, $i+j\in 2\mathbb{Z}$. And $j\neq\frac{k}{2}$, $i+j\geq\frac{k}{2}\geq |i-j|$, we have
\begin{eqnarray}\begin{split}
&M^{i,i^{'}}\boxtimes W(k,j)^{+}=M^{i,i^{'}}\boxtimes W(k,j)^{-}\\
&=\sum\limits_{\tiny{\begin{split}|i-j|\leq l\leq i+j \\  i+j+l\in 2\mathbb{Z} \ \ \ \\ i+j+l\leq 2k\ \ \ \end{split}}}\Big(W(k,l)^{+}+W(k,l)^{-}\Big)+ \Big(\widetilde{W(k,\frac{k}{2})}^{+}+\widetilde{W(k,\frac{k}{2})}^{-}\Big).
\label{eq:5.12'}\end{split}
\end{eqnarray}

If $k\in4\Z+2$, $i+j\in 2\mathbb{Z}$, or $k\in4\Z$, $i+j\in 2\mathbb{Z}+1$. And $j\neq\frac{k}{2}$, we have
\begin{eqnarray}\begin{split}
&M^{i,i^{'}}\boxtimes W(k,j)^{+}=M^{i,i^{'}}\boxtimes W(k,j)^{-}\\
&=\sum\limits_{\tiny{\begin{split}|i-j|\leq l\leq i+j \\  i+j+l\in 2\mathbb{Z} \ \ \ \\ i+j+l\leq 2k\ \ \ \end{split}}}\Big(W(k,l)^{+}+W(k,l)^{-}\Big).
\label{eq:5.14}\end{split}
\end{eqnarray}


 If $k\in 2\Z$, $i\in 2\mathbb{Z}+1$, we have
\begin{eqnarray}\begin{split}
&M^{i,i^{'}}\boxtimes W(k,\frac{k}{2})^{+}=M^{i,i^{'}}\boxtimes W(k,\frac{k}{2})^{-}\\
&=\sum\limits_{\tiny{\begin{split}|i-\frac{k}{2}|\leq l<\frac{k}{2} \\ i+\frac{k}{2}+l\in 2\mathbb{Z} \ \  \\ i+l\leq \frac{3k}{2}\ \ \ \end{split}}}  \Big(W(k,l)^{+}+W(k,l)^{-}\Big),
\label{eq:5.15}\end{split}
\end{eqnarray}

\begin{eqnarray}\begin{split}
&M^{i,i^{'}}\boxtimes \widetilde{W(k,\frac{k}{2})}^{+}=M^{i,i^{'}}\boxtimes \widetilde{W(k,\frac{k}{2})}^{-}\\
&=\sum\limits_{\tiny{\begin{split}|i-\frac{k}{2}|\leq l<\frac{k}{2} \\ i+\frac{k}{2}+l\in 2\mathbb{Z} \ \  \\ i+l\leq \frac{3k}{2}\ \ \ \end{split}}}  \Big(W(k,l)^{+}+W(k,l)^{-}\Big).
\label{eq:5.15'}\end{split}
\end{eqnarray}

  If $k\in 2\Z$, $i\in 2\mathbb{Z}$, $i^{'}\in 2\mathbb{Z}+1$, we have
\begin{eqnarray}\begin{split}
&M^{i,i^{'}}\boxtimes W(k,\frac{k}{2})^{+}=M^{i,i^{'}}\boxtimes W(k,\frac{k}{2})^{-}\\
&=\sum\limits_{\tiny{\begin{split}|i-\frac{k}{2}|\leq l< \frac{k}{2} \\  i+\frac{k}{2}+l\in 2\mathbb{Z} \ \  \\ i+l\leq \frac{3k}{2}\  \ \ \end{split}}}  \Big(W(k,l)^{+}+W(k,l)^{-}\Big)+\Big(\widetilde{W(k,\frac{k}{2})}^{+}+\widetilde{W(k,\frac{k}{2})}^{-}\Big).
\label{eq:5.29.}\end{split}
\end{eqnarray}

If $k\in 2\Z$, $i\in 2\mathbb{Z}$, $i^{'}\in 2\mathbb{Z}$, we have
\begin{eqnarray}\begin{split}
&M^{i,i^{'}}\boxtimes W(k,\frac{k}{2})^{+}=M^{i,i^{'}}\boxtimes W(k,\frac{k}{2})^{-}\\
&=\sum\limits_{\tiny{\begin{split}|i-\frac{k}{2}|\leq l< \frac{k}{2} \\  i+\frac{k}{2}+l\in 2\mathbb{Z} \ \  \\ i+l\leq \frac{3k}{2} \ \ \
 \end{split}}} \Big(W(k,l)^{+}+W(k,l)^{-}\Big)+\Big(W(k,\frac{k}{2})^{+}+W(k,\frac{k}{2})^{-}\Big).
\label{eq:5.30.}\end{split}
\end{eqnarray}

If $k\in 2\Z$, $i\in 2\mathbb{Z}$, $i^{'}\in 2\mathbb{Z}+1$, we have
\begin{eqnarray}\begin{split}
&M^{i,i^{'}}\boxtimes \widetilde{W(k,\frac{k}{2})}^{+}=M^{i,i^{'}}\boxtimes \widetilde{W(k,\frac{k}{2})}^{-}\\
&=\sum\limits_{\tiny{\begin{split}|i-\frac{k}{2}|\leq l<\frac{k}{2} \\  i+\frac{k}{2}+l\in 2\mathbb{Z} \ \  \\ i+l\leq \frac{3k}{2}\  \ \ \end{split}}}  \Big(W(k,l)^{+}+W(k,l)^{-}\Big)+\Big(W(k,\frac{k}{2})^{+}+W(k,\frac{k}{2})^{-}\Big).
\label{eq:5.31.}\end{split}
\end{eqnarray}

If $k\in 2\Z$, $i\in 2\mathbb{Z}$, $i^{'}\in 2\mathbb{Z}$, we have
\begin{eqnarray}\begin{split}
&M^{i,i^{'}}\boxtimes \widetilde{W(k,\frac{k}{2})}^{+}=M^{i,i^{'}}\boxtimes \widetilde{W(k,\frac{k}{2})}^{-}\\
&=\sum\limits_{\tiny{\begin{split}|i-\frac{k}{2}|\leq l< \frac{k}{2} \\  i+\frac{k}{2}+l\in 2\mathbb{Z} \ \  \\ i+l\leq \frac{3k}{2} \ \ \
 \end{split}}} \Big(W(k,l)^{+}+W(k,l)^{-}\Big)+\Big(\widetilde{W(k,\frac{k}{2})}^{+}+\widetilde{W(k,\frac{k}{2})}^{-}\Big).
\label{eq:5.32.}\end{split}
\end{eqnarray}

(4) If $k\in 2\Z$, $j\in 2\mathbb{Z}+1$, $j\neq\frac{k}{2}$, we have
\begin{eqnarray}\begin{split}
&(M^{\frac{k}{2},0})^{+}\boxtimes W(k,j)^{\pm}=(M^{\frac{k}{2},0})^{-}\boxtimes W(k,j)^{\pm}\\
&=\sum\limits_{\tiny{\begin{split}|\frac{k}{2}-j|\leq l<\frac{k}{2} \\ \frac{k}{2}+j+l\in 2\mathbb{Z} \  \\ j+l\leq \frac{3k}{2}\ \ \ \end{split}}}\Big(W(k,l)^{+}+W(k,l)^{-}\Big).
\label{eq:5.23}\end{split}
\end{eqnarray}

If $k\in 2\Z$, $j\in 2\mathbb{Z}$, $j\neq\frac{k}{2}$, we have
\begin{eqnarray}
(M^{\frac{k}{2},0})^{+}\boxtimes W(k,j)^{\pm}=\sum\limits_{\tiny{\begin{split}|\frac{k}{2}-j|\leq l<\frac{k}{2} \\ \frac{k}{2}+j+l\in 2\mathbb{Z} \  \\ j+l\leq \frac{3k}{2}\ \ \ \end{split}}} \Big(W(k,l)^{+}+W(k,l)^{-}\Big)+\Big(W(k,\frac{k}{2})^{\pm}+\widetilde{W(k,\frac{k}{2})}^{\pm}\Big),
\label{eq:5.24}
\end{eqnarray}
\begin{eqnarray}
(M^{\frac{k}{2},0})^{-}\boxtimes W(k,j)^{\pm}=\sum\limits_{\tiny{\begin{split}|\frac{k}{2}-j|\leq l<\frac{k}{2} \\ \frac{k}{2}+j+l\in 2\mathbb{Z} \  \\ j+l\leq \frac{3k}{2}\ \ \ \end{split}}}  \Big(W(k,l)^{+}+W(k,l)^{-}\Big)+\Big(W(k,\frac{k}{2})^{\mp}+\widetilde{W(k,\frac{k}{2})}^{\mp}\Big).
\label{eq:5.24'}
\end{eqnarray}

If $k\in 4\Z+2$, we have
\begin{eqnarray}
(M^{\frac{k}{2},0})^{+}\boxtimes W(k,\frac{k}{2})^{\pm}=\sum\limits_{\tiny{\begin{split}0\leq l\leq \frac{k}{2}-1 \\  k+l\in 2\mathbb{Z} \  \\ l\leq k\ \ \ \end{split}}}W(k,l)^{\mbox{sign}(\frac{k}{2},\frac{k}{2},l)^{\pm}}.
\label{eq:5.25}
\end{eqnarray}
\begin{eqnarray}
(M^{\frac{k}{2},0})^{-}\boxtimes W(k,\frac{k}{2})^{\pm}=\sum\limits_{\tiny{\begin{split}0\leq l\leq \frac{k}{2}-1 \\  k+l\in 2\mathbb{Z} \  \\ l\leq k\ \ \ \end{split}}} W(k,l)^{\mbox{sign}(\frac{k}{2},\frac{k}{2},l)^{\mp}}.
\label{eq:5.25'}
\end{eqnarray}
\begin{eqnarray}
(M^{\frac{k}{2},0})^{+}\boxtimes \widetilde{W(k,\frac{k}{2})}^{\pm}=\sum\limits_{\tiny{\begin{split}0\leq l\leq \frac{k}{2}-1 \\  k+l\in 2\mathbb{Z} \  \\ l\leq k\ \ \ \end{split}}} W(k,l)^{\mbox{sign}(\frac{k}{2},\frac{k}{2},l)^{\pm}}.
\label{eq:5.26}
\end{eqnarray}
\begin{eqnarray}
(M^{\frac{k}{2},0})^{-}\boxtimes \widetilde{W(k,\frac{k}{2})}^{\pm}=\sum\limits_{\tiny{\begin{split}0\leq l\leq \frac{k}{2}-1 \\  k+l\in 2\mathbb{Z} \  \\ l\leq k\ \ \ \end{split}}}  W(k,l)^{\mbox{sign}(\frac{k}{2},\frac{k}{2},l)^{\mp}}.
\label{eq:5.26'}
\end{eqnarray}
 If $k\in 4\mathbb{Z}$, we have
\begin{eqnarray*}
(M^{\frac{k}{2},0})^{+}\boxtimes W(k,\frac{k}{2})^{\pm}=\sum\limits_{\tiny{\begin{split}0\leq l\leq \frac{k}{2}-1 \\ k+l\in 2\mathbb{Z} \ \  \\ l\leq k \ \ \ \ \ \end{split}}} W(k,l)^{\mbox{sign}(\frac{k}{2},\frac{k}{2},l)^{\pm}}+W(k,\frac{k}{2})^{\pm}.
\end{eqnarray*}

\begin{eqnarray*}
(M^{\frac{k}{2},0})^{-}\boxtimes W(k,\frac{k}{2})^{\pm}=\sum\limits_{\tiny{\begin{split}0\leq l\leq \frac{k}{2}-1 \\ k+l\in 2\mathbb{Z} \ \  \\ l\leq k  \ \ \ \ \ \end{split}}}  W(k,l)^{\mbox{sign}(\frac{k}{2},\frac{k}{2},l)^{\mp}}+W(k,\frac{k}{2})^{\mp}.
\end{eqnarray*}

\begin{eqnarray*}
(M^{\frac{k}{2},0})^{+}\boxtimes \widetilde{W(k,\frac{k}{2})}^{\pm}=\sum\limits_{\tiny{\begin{split}0\leq l\leq \frac{k}{2}-1 \\ k+l\in 2\mathbb{Z} \ \  \\ l\leq k \ \ \ \ \ \end{split}}} W(k,l)^{\mbox{sign}(\frac{k}{2},\frac{k}{2},l)^{\pm}}+\widetilde{W(k,\frac{k}{2})}^{\pm}.
\end{eqnarray*}

\begin{eqnarray*}
(M^{\frac{k}{2},0})^{-}\boxtimes \widetilde{W(k,\frac{k}{2})}^{\pm}=\sum\limits_{\tiny{\begin{split}0\leq l\leq \frac{k}{2}-1 \\ k+l\in 2\mathbb{Z} \ \  \\ l\leq k  \ \ \ \ \ \end{split}}}  W(k,l)^{\mbox{sign}(\frac{k}{2},\frac{k}{2},l)^{\mp}}+\widetilde{W(k,\frac{k}{2})}^{\mp}.
\end{eqnarray*}

\end{thm}

\begin{proof} We will prove the case for $k\in 2\Z+1$ and $k\in 4\Z+2$,  the proof of the case $k\in 4\Z$ is similar to the proof of the case $k\in 4\Z+2$. Note that $(M^{k,0})^{+}=K_{0}^{+}$, and we have the intertwining operator in Lemma \ref{lem:intertwining.}. Similar to the proof of
(\ref{eq:5.1}) and (\ref{eq:5.1'}) in Theorem \ref{para-fusion-untwist}, we can obtain (\ref{eq:5.9}), (\ref{eq:5.9'}), (\ref{eq:5.9.})-(\ref{eq:.5.9.'}).

For (\ref{eq:5.10}) and (\ref{eq:5.10'}), from  Theorem \ref{fusion.aff.}, we have
\begin{eqnarray}\label{fusion.twist1'}
L(k,i)^{+}\boxtimes \overline{L(k,j)}^{+}=\sum\limits_{\tiny{\begin{split}|i-j|\leq l\leq i+j \\  i+j+l\in 2\mathbb{Z} \ \ \ \\ i+j+l\leq 2k\ \ \ \end{split}}} \overline{L(k,l)}^{\mbox{sign}(i,j,l)^{+}}. \label{eq:5.43.}
\end{eqnarray}
From the decomposition (\ref{eq:4.0.}):
\begin{eqnarray*}
L(k,i)=\bigoplus_{j=0}^{k-1}V_{\mathbb{Z}\gamma+(i-2j)\gamma/2k}\otimes M^{i,j} \ \ \ \mbox{for} \ 0\leq i\leq k,
\end{eqnarray*}
we have $$V_{\mathbb{Z}\gamma}^{+}\otimes (M^{i,\frac{i}{2}})^{+}\subseteq L(k,i)^{+}.$$
From the decomposition (\ref{eq:4.9.}): \begin{eqnarray*}
\overline{L(k,i)}=V_{\mathbb{Z}\gamma}^{T_{a_{i}}}\otimes W(k,i) \ \  \mbox{for} \ i\neq\frac{k}{2},
\end{eqnarray*}
where $a_{i}=1$ or $2$ depending on $i$, we have $$(V_{\mathbb{Z}\gamma}^{T_{a_{i}}})^{+}\otimes W(k,j)^{+}\subseteq \overline{L(k,i)}^{+}.$$
 Since $(M^{i,\frac{i}{2}})^{+}\subseteq L(k,i)^{+}$, $W(k,j)^{+}\subseteq \overline{L(k,i)}^{+}$, and $V_{\mathbb{Z}\gamma}^{+}\boxtimes_{V_{\mathbb{Z}\gamma}^{+}}(V_{\mathbb{Z}\gamma}^{T_{a_{i}}})^{+}=(V_{\mathbb{Z}\gamma}^{T_{a_{i}}})^{+}$, by using the quantum dimension obtained in Theorem \ref{quantum-dimension}:
\begin{eqnarray*}
 \mbox{qdim}_{K_{0}^{\sigma}}(M^{i,\frac{i}{2}})^{+}=\frac{\sin\frac{\pi(i+1)}{k+2}}{\sin\frac{\pi}{k+2}},  \
\mbox{qdim}_{K_{0}^{\sigma}}W(k,j)^{+}
=\sqrt{k}\frac{\sin\frac{\pi(j+1)}{k+2}}{\sin\frac{\pi}{k+2}} \ \mbox{for}\ j\neq\frac{k}{2},\end{eqnarray*}
and
\begin{eqnarray*}\mbox{qdim}_{K_{0}^{\sigma}}\Big((M^{i,\frac{i}{2}})^{+}\boxtimes W(k,j)^{+}\Big)=
\mbox{qdim}_{K_{0}^{\sigma}}(M^{i,\frac{i}{2}})^{+}\cdot\mbox{qdim}_{K_{0}^{\sigma}}W(k,j)^{+},
\end{eqnarray*} together with (\ref{eq:5.43.}), we can deduce
\begin{eqnarray*}
(M^{i,\frac{i}{2}})^{+}\boxtimes W(k,j)^{+}=\sum\limits_{\tiny{\begin{split}|i-j|\leq l\leq i+j \\  i+j+l\in 2\mathbb{Z} \ \ \ \\ i+j+l\leq 2k\ \ \ \end{split}}} W(k,l)^{\mbox{sign}(i,j,l)^{+}},
\end{eqnarray*}
where we notice that $l\neq \frac{k}{2}$ in this case, then (\ref{eq:5.10}), (\ref{eq:5.10'}) follows immediately. By the similar proof to (\ref{eq:5.10}) and (\ref{eq:5.10'}), just noticing the definition of $\widetilde{W(k,\frac{k}{2})}^{\pm}$, we can get (\ref{eq:5.16})-(\ref{eq:5.18'}).

For (\ref{eq:5.19}), similar to the arguments in the proof of (\ref{eq:5.10}), but noticing that in this case $l$ can take $\frac{k}{2}$, and
\begin{eqnarray*}
 \mbox{qdim}_{K_{0}^{\sigma}}W(k,\frac{k}{2})^{+}=\frac{\sqrt{k}}{2}\frac{\sin\frac{\pi(\frac{k}{2}+1)}{k+2}}{\sin\frac{\pi}{k+2}}, \end{eqnarray*}
we can obtain that
\begin{eqnarray}
(M^{i,\frac{i}{2}})^{+}\boxtimes W(k,\frac{k}{2})^{+}=\sum\limits_{\tiny{\begin{split}|i-\frac{k}{2}|\leq l<\frac{k}{2} \\ i+\frac{k}{2}+l\in 2\mathbb{Z} \  \\ i+l\leq \frac{3k}{2}\ \ \ \end{split}}}W(k,l)^{\mbox{sign}(i,\frac{k}{2},l)^{+}}+M. \label{eq:5..1}
\end{eqnarray}
Since $\mbox{sign}(i,\frac{k}{2},\frac{k}{2})^{+}=-$, from the definition of $\widetilde{W(k,\frac{k}{2})}^{+}$, we can deduce that $M=\widetilde{W(k,\frac{k}{2})}^{+}$ or $M=W(k,\frac{k}{2})^{-}$. We now prove that
$M=\widetilde{W(k,\frac{k}{2})}^{+}$.
From the lattice realization of the irreducible $K_0$-modules $M^{i,j}$, i.e., (\ref{eq:3.3}) and (\ref{eq:3.4.}), we know that there exists
$m\in \Z$ such that

\begin{eqnarray*}
v^{i,\frac{i}{2}-1}(m)v^{\frac{k}{2},\frac{k}{2}}=a_{i}e(0)v^{\frac{k}{2},\frac{k}{2}}, \ \ \ v^{i,\frac{i}{2}+1}(m-1)v^{\frac{k}{2},\frac{k}{2}}=b_{i}f(-1)v^{\frac{k}{2},\frac{k}{2}}
\end{eqnarray*}
for some nonzero complex numbers $a_i$ and $b_i$.
This implies that
\begin{eqnarray*}
(e^{'}(0)^{\frac{i}{2}+1}\eta_{i})(m)\eta_{\frac{k}{2}}=a_{i}e^{'}(0)\eta_{\frac{k}{2}},\ \ \ (e^{'}(0)^{\frac{i}{2}-1}\eta_{i})(m-2)\eta_{\frac{k}{2}}=b_{i}f^{'}(-1)\eta_{\frac{k}{2}}.
\end{eqnarray*}
Here we use an identification of basis $\{h,\ e,\ f\}$ and $\{h^{'},\ e^{'},\ f^{'}\}$. We can also deduce from the lowest weight that \begin{eqnarray}
(e^{'}(0)^{j}\eta_{i})(n)\eta_{\frac{k}{2}}=0 \label{eq:5.70}
\end{eqnarray}
 for $n>m$. Note that in this case $\frac{i}{2}\in 2\Z+1$, $v^{i,\frac{i}{2}}$ is a linear combination of the vectors $\eta_{i}, \ h^{2}(0)\eta_{i},\cdots,$ $h^{\frac{i}{2}-1}(0)\eta_{i}, \ h^{\frac{i}{2}+1}(0)\eta_{i},\cdots,h^{i}(0)\eta_{i},$ by straightforward calculations. It shows that $v^{i,\frac{i}{2}}\in L(k,i)^{+}$, and from the discussion above, we know that $\left(\begin{array}{c}
{\overline{L(k,\frac{k}{2})}}^{-}\\
L(k,i)^{+} \ {\overline{L(k,\frac{k}{2})}^{+}}
\end{array}\right)\neq 0$. From Lemma \ref{lem:intertwining.}, we have
\begin{eqnarray*}\widetilde{\mathcal{Y}}(e^{'}(0)^{j}\eta_{i},z)=\mathcal{Y}(\Delta(h^{''},z)e^{'}(0)^{j}\eta_{i},z)
=z^{\frac{j}{2}-\frac{i}{4}}\mathcal{Y}(e^{'}(0)^{j}\eta_{i},z).
\end{eqnarray*}
Thus we have $(e^{'}(0)^{\frac{i}{2}+1}\eta_{i})_{m-\frac{1}{2}}=(e^{'}(0)^{\frac{i}{2}+1}\eta_{i})(m)$. Together with (\ref{eq:5.70}) and by considering the weights of the lattice realization, we obtain that for $j>1, j\in 2\Z+1$,
\begin{eqnarray}(e^{'}(0)^{\frac{i}{2}+j}\eta_{i})_{m-\frac{1}{2}}\eta_{\frac{k}{2}}=
(e^{'}(0)^{\frac{i}{2}+j}\eta_{i})({m+\frac{j}{2}-\frac{1}{2}})\eta_{\frac{k}{2}}=0, \label{eq:5.71}
\end{eqnarray}
\begin{eqnarray}(e^{'}(0)^{\frac{i}{2}-j}\eta_{i})_{m-\frac{1}{2}}\eta_{\frac{k}{2}}=
(e^{'}(0)^{\frac{i}{2}-j}\eta_{i})({m-\frac{j}{2}-\frac{1}{2}})\eta_{\frac{k}{2}}=0. \label{eq:5.72}
\end{eqnarray}
Let \begin{eqnarray*}v^{i,\frac{i}{2}}=\sum\limits_{\tiny{\begin{split} 0\leq j\leq i \\ j\in 2\mathbb{Z} \   \end{split}}}c_{j}e^{'}(0)^{j}\eta_{i},
\end{eqnarray*}
then from (\ref{eq:5.71}) and (\ref{eq:5.72}), we have
$$v^{i,\frac{i}{2}}_{m-\frac{1}{2}}\eta_{\frac{k}{2}}=c_{\frac{i}{2}+1}a_{i}e^{'}(0)\eta_{\frac{k}{2}}+c_{\frac{i}{2}-1}b_{i}f^{'}(-1)\eta_{\frac{k}{2}}.$$
Note that $$\overline{L(k,\frac{k}{2})}^{-}=V_{\Z{\gamma}}^{T_{a_{\frac{k}{2}}},+}\otimes W(k,\frac{k}{2})^{-}\oplus V_{\Z{\gamma}}^{T^{'}_{a_{\frac{k}{2}}},+}\otimes \widetilde{W(k,\frac{k}{2})}^{+}\oplus V_{\Z{\gamma}}^{T_{a_{\frac{k}{2}}},-}\otimes W(k,\frac{k}{2})^{+}
\oplus V_{\Z{\gamma}}^{T^{'}_{a_{\frac{k}{2}}},-}\otimes \widetilde{W(k,\frac{k}{2})}^{-},$$
 $v^{i,\frac{i}{2}}\in (M^{i,\frac{i}{2}})^{+}$,  $\left(\begin{array}{c}
{\overline{L(k,\frac{k}{2})}}^{-}\\
L(k,i)^{+} \ {\overline{L(k,\frac{k}{2})}^{+}}
\end{array}\right)\neq 0$, and $$(e-f)_{-\frac{1}{2}}\eta_{\frac{k}{2}}=(f^{'}-e^{'})_{-\frac{1}{2}}\eta_{\frac{k}{2}}=(f^{'}(-1)-e^{'}(0))\eta_{\frac{k}{2}}\in \overline{L(k,\frac{k}{2})}^{-}_{\frac{1}{2}},$$ we deduce that $$v^{i,\frac{i}{2}}_{m-\frac{1}{2}}\eta_{\frac{k}{2}}=c(f^{'}(-1)-e^{'}(0))\eta_{\frac{k}{2}}$$ for some nonzero complex number $c$, which means that
$\left(\begin{array}{c}
{\widetilde{{W(k,\frac{k}{2})}}}^{+}\\
(M^{i,\frac{i}{2}})^{+} \ W(k,\frac{k}{2})^{+}
\end{array}\right)\neq 0$, that is, $M=\widetilde{W(k,\frac{k}{2})}^{+}$ as required. Thus we have (\ref{eq:5.19}).
Then (\ref{eq:5.19'}) follows immediately. Similarly, we can prove (\ref{eq:5.20}) and (\ref{eq:5.20'}).

For (\ref{eq:5.21}), similar to the analysis of (\ref{eq:5.19}), in this case, we need to prove that $M$ is $W(k,\frac{k}{2})^{+}$ in (\ref{eq:5..1}). By applying the lattice realization of $K_0$-module $M^{i,j}$, we can obtain that there exists $m\in\Z$ such that
\begin{eqnarray}
(e^{'}(0)^{\frac{i}{2}}\eta_{i})(m)\eta_{\frac{k}{2}}=a_{i}\eta_{\frac{k}{2}},  \label{eq:5..3}
\end{eqnarray}
for some nonzero complex number $a_i$. By analyzing the weights in $\overline{L(k,\frac{k}{2})}^{+}$, we can get
\begin{eqnarray}
(e^{'}(0)^{j}\eta_{i})(m)\eta_{\frac{k}{2}}=0 \label{eq:5..4}
\end{eqnarray}
for $j\in 2\Z, \ j\neq\frac{i}{2}$. Similar to the proof of (\ref{eq:5.19}), and noticing that in this case $v^{i,\frac{i}{2}}$ is a linear combination of vectors $\eta_{i}, \ ,e^{'}(0)^{2}\eta_{i},\cdots,e^{'}(0)^{\frac{i}{2}}\eta_{i},\cdots,e^{'}(0)^{i}\eta_{i}$, i.e.,
we may write \begin{eqnarray*}v^{i,\frac{i}{2}}=\sum\limits_{\tiny{\begin{split} 0\leq j\leq i \\ j\in 2\mathbb{Z} \   \end{split}}}c_{j}e^{'}(0)^{j}\eta_{i}
\end{eqnarray*}
with $c_{j}\neq 0$ for $j\in 2\Z, 0\leq j\leq i$. Thus from (\ref{eq:5..3}) and (\ref{eq:5..4}), we have
\begin{eqnarray*}
(v^{i,\frac{i}{2}})(m)\eta_{\frac{k}{2}}=c_{\frac{i}{2}}a_{i}\eta_{\frac{k}{2}}\neq 0.
\end{eqnarray*}
which means that
$\left(\begin{array}{c}
{W(k,\frac{k}{2})}^{+}\\
(M^{i,\frac{i}{2}})^{+} \ W(k,\frac{k}{2})^{+}
\end{array}\right)\neq 0$, that is, $M=W(k,\frac{k}{2})^{+}$ as required. Thus we have (\ref{eq:5.21}), and (\ref{eq:5.21'}) follows immediately.
Similarly, we can prove (\ref{eq:5.22})-(\ref{eq:5.22'}).

For (\ref{eq:5.11}), since $M^{i,i^{'}}$ are untwisted modules of type $II$, they are irreducible as $K_{0}^{\sigma}$-modules. This shows that
\begin{eqnarray*}
\begin{split}
M^{i,i^{'}}\boxtimes W(k,j)^{+}&=K_{0}^{+}\boxtimes M^{i,i^{'}}\boxtimes W(k,j)^{+}\\
&=K_{0}^{-}\boxtimes M^{i,i^{'}}\boxtimes W(k,j)^{+}\\
&=M^{i,i^{'}}\boxtimes W(k,j)^{-},
\end{split}
\end{eqnarray*}
since from Theorem \ref{quantum-dimension}, we have
\begin{eqnarray*}
 \mbox{qdim}_{K_{0}^{\sigma}}M^{i,i^{'}}=2\frac{\sin\frac{\pi(i+1)}{k+2}}{\sin\frac{\pi}{k+2}},  \
\mbox{qdim}_{K_{0}^{\sigma}}W(k,j)^{+}
=\sqrt{k}\frac{\sin\frac{\pi(j+1)}{k+2}}{\sin\frac{\pi}{k+2}} \ \mbox{for} \ j\neq\frac{k}{2}.\end{eqnarray*}
\begin{eqnarray*}
 \mbox{qdim}_{K_{0}^{\sigma}}W(k,\frac{k}{2})^{+}= \mbox{qdim}_{K_{0}^{\sigma}}\widetilde{W(k,\frac{k}{2})}^{+}=\frac{\sqrt{k}}{2}\frac{\sin\frac{\pi(\frac{k}{2}+1)}{k+2}}{\sin\frac{\pi}{k+2}}.  \
\end{eqnarray*}

By using
\begin{eqnarray*}\mbox{qdim}_{K_{0}^{\sigma}}\Big((M^{i,i^{'}})^{+}\boxtimes W(k,j)^{+}\Big)=
\mbox{qdim}_{K_{0}^{\sigma}}(M^{i,i^{'}})^{+}\cdot\mbox{qdim}_{K_{0}^{\sigma}}W(k,j)^{+},
\end{eqnarray*}
and noticing that all the twisted type  modules of $K_{0}^{\sigma}$ are constructed from the twisted type modules of the affine vertex operator algebra\cite{JW}, together with Lemma \ref{lem:intertwining.}, we can get
 that in this case $l\neq\frac{k}{2}$, (\ref{eq:5.11}) holds. By the same reason as in the proof of (\ref{eq:5.11}), we can obtain (\ref{eq:5.12})-(\ref{eq:5.14}), just noticing that in (\ref{eq:5.12}) and (\ref{eq:5.14}), $l\neq\frac{k}{2}$.

For (\ref{eq:5.15}), from Theorem \ref{quantum-dimension} we notice that on the left side of the equation (\ref{eq:5.15}),  the quantum dimension is
\begin{eqnarray*}\begin{split} \mbox{qdim}_{K_{0}^{\sigma}}\Big(M^{i,i^{'}}\boxtimes W(k,\frac{k}{2})^{+}\Big)
&=\mbox{qdim}_{K_{0}^{\sigma}}M^{i,i^{'}}\cdot\mbox{qdim}_{K_{0}^{\sigma}}W(k,\frac{k}{2})^{+}\\
&=\sqrt{k}\frac{\sin\frac{\pi(i+1)}{k+2}}{\sin\frac{\pi}{k+2}}\frac{\sin\frac{\pi(\frac{k}{2}+1)}{k+2}}{\sin\frac{\pi}{k+2}},
\end{split}
\end{eqnarray*}
and note that if $i\leq\frac{k}{2}$, then $l_{\tiny\mbox{min}}=\frac{k}{2}-i, \ l_{\tiny\mbox{max}}=\frac{k}{2}+i$. Thus
$$\sum\limits_{\frac{k}{2}-i\leq l<\frac{k}{2}}\frac{\sin\frac{\pi(l+1)}{k+2}}{\sin\frac{\pi}{k+2}}=\sum\limits_{\frac{k}{2}< l\leq i+\frac{k}{2}}\frac{\sin\frac{\pi(l+1)}{k+2}}{\sin\frac{\pi}{k+2}}.$$
If $i>\frac{k}{2}$, then $l_{\tiny\mbox{min}}=i-\frac{k}{2}, \ l_{\tiny\mbox{max}}=2k-i-\frac{k}{2}=\frac{3k}{2}-i$. Thus
$$\sum\limits_{i-\frac{k}{2}\leq l<\frac{k}{2}}\frac{\sin\frac{\pi(l+1)}{k+2}}{\sin\frac{\pi}{k+2}}=\sum\limits_{\frac{k}{2}< l\leq \frac{3k}{2}-i}\frac{\sin\frac{\pi(l+1)}{k+2}}{\sin\frac{\pi}{k+2}}.$$
Then we can get (\ref{eq:5.15}) by using the fact that the quantum dimension is equal on both sides of the equation. Similarly, we can get (\ref{eq:5.15'}).

For (\ref{eq:5.29.}) and (\ref{eq:5.30.}), we divide the proof into four cases (i) $i\in 4\Z+2$, $i^{'}-\frac{i}{2}\in2\Z$, (ii)$i\in 4\Z+2$, $i^{'}-\frac{i}{2}\in2\Z+1$, (iii)$i\in 4\Z$, $i^{'}-\frac{i}{2}\in2\Z$, (iv) $i\in 4\Z$, $i^{'}-\frac{i}{2}\in2\Z+1$. If $i\in4\Z+2$, similar to the arguments in the proof of (\ref{eq:5.11}), but noticing that in this case $l$ can take $\frac{k}{2}$, and
\begin{eqnarray*}
 \mbox{qdim}_{K_{0}^{\sigma}}W(k,\frac{k}{2})^{+}=\frac{\sqrt{k}}{2}\frac{\sin\frac{\pi(\frac{k}{2}+1)}{k+2}}{\sin\frac{\pi}{k+2}}, \end{eqnarray*}
we obtain that
\begin{eqnarray}\begin{split}
&M^{i,i^{'}}\boxtimes W(k,\frac{k}{2})^{+}=M^{i,i^{'}}\boxtimes W(k,\frac{k}{2})^{-}\\
&=\sum\limits_{\tiny{\begin{split}|i-\frac{k}{2}|\leq l\leq \frac{k}{2} \\  i+\frac{k}{2}+l\in 2\mathbb{Z} \ \  \\ i+l\leq \frac{3k}{2}\  \ \ \end{split}}}  \Big(W(k,l)^{+}+W(k,l)^{-}\Big)+(M^{+}+M^{-}).
\label{eq:.5.29..}\end{split}
\end{eqnarray}
We prove that $M=\widetilde{W(k,\frac{k}{2})}$ if $i^{'}-\frac{i}{2}\in 2\Z$ and $M=W(k,\frac{k}{2})$ if $i^{'}-\frac{i}{2}\in 2\Z+1$.
Since
\begin{eqnarray*}
h^{'}(0)v^{i,\frac{i}{2}}=(e+f)(0)v^{i,\frac{i}{2}}=(\frac{i}{2}+1)(v^{i,\frac{i}{2}-1}+v^{i,\frac{i}{2}+1}),
\end{eqnarray*}
and noticing that $M^{i,i^{'}}\cong M^{i,i-i^{'}}$ as $K_0^{\sigma}$-module, we may assume $i^{'}>\frac{i}{2}$. By induction, we can get if $i^{'}-\frac{i}{2}\in 2\Z$,
\begin{eqnarray*}
h^{'}(0)^{i^{'}-\frac{i}{2}}v^{i,\frac{i}{2}}\in M^{i,i^{'}}\oplus M^{i,i^{'}-2}\oplus \cdots\oplus M^{i,\frac{i}{2}},
\end{eqnarray*}
and if $i^{'}-\frac{i}{2}\in 2\Z+1$,
\begin{eqnarray*}
h^{'}(0)^{i^{'}-\frac{i}{2}}v^{i,\frac{i}{2}}\in M^{i,i^{'}}\oplus M^{i,i^{'}-2}\oplus \cdots\oplus M^{i,\frac{i}{2}+1}.
\end{eqnarray*}
From the proof of (\ref{eq:5.19}), we know that
\begin{eqnarray*}
(v^{i,\frac{i}{2}})_{m-\frac{1}{2}}\eta_{\frac{k}{2}}=
(c_{\frac{i}{2}+1}e^{'}(0)^{\frac{i}{2}+1}\eta_{i}+c_{\frac{i}{2}-1}e^{'}(0)^{\frac{i}{2}-1}\eta_{i})_{m-\frac{1}{2}}\eta_{\frac{k}{2}}
=c(e^{'}(0)-f^{'}(-1))\eta_{\frac{k}{2}}\neq 0.
\end{eqnarray*}
 Noticing that
\begin{eqnarray*}
h(0)^{i^{'}-\frac{i}{2}}v^{i,\frac{i}{2}-1}=2^{i^{'}-\frac{i}{2}}v^{i,\frac{i}{2}-1},
\end{eqnarray*}

\begin{eqnarray*}
h(0)^{i^{'}-\frac{i}{2}}v^{i,\frac{i}{2}+1}=(-2)^{i^{'}-\frac{i}{2}}v^{i,\frac{i}{2}+1},
\end{eqnarray*}
together with the relation of the intertwining operator among untwisted modules and the intertwining operators among twisted modules: $$\widetilde{\mathcal{Y}}(h^{'}(0)^{i^{'}-\frac{i}{2}}v^{i,\frac{i}{2}},z)=\mathcal{Y}(\Delta(h^{''},z)h^{'}(0)^{i^{'}-\frac{i}{2}}v^{i,\frac{i}{2}},z),$$
we can deduce that if $i^{'}-\frac{i}{2}\in 2\Z$,
\begin{eqnarray*}
(v^{i,i^{'}}+v^{i,i-i^{'}})_{m-\frac{1}{2}}\eta_{\frac{k}{2}}=A_{i,i^{'}}(e^{'}(0)-f^{'}(-1))\eta_{\frac{k}{2}}
=A_{i,i^{'}}(e-f)_{\frac{1}{2}}\eta_{\frac{k}{2}}\in \widetilde{W(k,\frac{k}{2})}^{+}.
\end{eqnarray*}
If $i^{'}-\frac{i}{2}\in 2\Z+1$,
\begin{eqnarray*}
(v^{i,i^{'}}+v^{i,i-i^{'}})_{m-\frac{1}{2}}\eta_{\frac{k}{2}}=A_{i,i^{'}}(e^{'}(0)+f^{'}(-1))\eta_{\frac{k}{2}}
=A_{i,i^{'}}h_{\frac{1}{2}}\eta_{\frac{k}{2}}\in W(k,\frac{k}{2})^{+}
\end{eqnarray*}
for a nonzero complex number $A_{i,i^{'}}$. That is, if $i^{'}-\frac{i}{2}\in 2\Z$, $\left(\begin{array}{c}
{\widetilde{{W(k,\frac{k}{2})}}}^{+}\\
M^{i,i^{'}} \ W(k,\frac{k}{2})^{+}
\end{array}\right)\neq 0$. Then $M=\widetilde{W(k,\frac{k}{2})}$ as required. And
if $i^{'}-\frac{i}{2}\in 2\Z+1$, $\left(\begin{array}{c}
W(k,\frac{k}{2})^{+}\\
M^{i,i^{'}} \ W(k,\frac{k}{2})^{+}
\end{array}\right)\neq 0$, that is, $M=W(k,\frac{k}{2})$ as required.

If $i\in 4\Z$, $i^{'}-\frac{i}{2}\in 2\Z$, similar to the arguments in the above discussion, we need to prove that
$M=W(k,\frac{k}{2})$. With the proof of (\ref{eq:5.21}), notice that in this case, \begin{eqnarray*}
(v^{i,\frac{i}{2}})(m)\eta_{\frac{k}{2}}=c_{\frac{i}{2}}a_{i}\eta_{\frac{k}{2}}\in W(k,\frac{k}{2})^{+}.
\end{eqnarray*}
Following the proof of (\ref{eq:5.29.}), we can deduce that  \begin{eqnarray}
(v^{i,i^{'}}+v^{i,i-i^{'}})(m)\eta_{\frac{k}{2}}=B_{i,i^{'}}\eta_{\frac{k}{2}}\in W(k,\frac{k}{2})^{+}, \label{eq:.5..5}
\end{eqnarray}
for a nonzero complex number $B_{i,i^{'}}$. Together with the relation of the intertwining operators among untwisted modules and the intertwining operator among twisted modules, we have $\left(\begin{array}{c}
W(k,\frac{k}{2})^{+}\\
M^{i,i^{'}} \ W(k,\frac{k}{2})^{+}
\end{array}\right)\neq 0$,  that is, $M=W(k,\frac{k}{2})$ as required.

If $i\in 4\Z$, $i^{'}-\frac{i}{2}\in 2\Z+1$, we need to prove $M=\widetilde{W(k,\frac{k}{2})}$. Suppose $M\neq \widetilde{W(k,\frac{k}{2})}$, i.e., $M=W(k,\frac{k}{2})$. Notice that in this case $i\in 4\Z$, and $v^{i,\frac{i}{2}+1}+v^{i,\frac{i}{2}-1}$ is a linear combination of vectors $e^{'}(0)\eta_{i}, \ ,e^{'}(0)^{3}\eta_{i},\cdots,
 e^{'}(0)^{\frac{i}{2}-1}\eta_{i}, \  e^{'}(0)^{\frac{i}{2}+1}\eta_{i},\cdots, e^{'}(0)^{i-1}\eta_{i}$, that is, we can write
 \begin{eqnarray*} v^{i,\frac{i}{2}+1}+v^{i,\frac{i}{2}-1}=\sum\limits_{\tiny{\begin{split} 1\leq j\leq i-1 \\ j\in 2\mathbb{Z}+1 \   \end{split}}}d_{j}e^{'}(0)^{j}\eta_{i}.
 \end{eqnarray*}
  As the proof of (\ref{eq:5.29.}), there exists $m\in \Z$ such that
\begin{eqnarray*} (e^{'}(0)^{\frac{i}{2}+1}\eta_{i})(m)\eta_{\frac{k}{2}}=a_{i}e^{'}(0)\eta_{\frac{k}{2}},
 \end{eqnarray*}
\begin{eqnarray*} (e^{'}(0)^{\frac{i}{2}-1}\eta_{i})(m-1)\eta_{\frac{k}{2}}=b_{i}f^{'}(-1)\eta_{\frac{k}{2}},
 \end{eqnarray*}
 for some nonzero complex numbers $a_{i}, b_{i}$.
 If $M=W(k,\frac{k}{2})$, this shows that
\begin{eqnarray*} (e^{'}(0)+f^{'}(-1))\eta_{\frac{k}{2}}=h_{-\frac{1}{2}}\eta_{\frac{k}{2}}\in {\mathbb{C}} (v^{i,\frac{i}{2}+1}+v^{i,\frac{i}{2}-1})_{m-\frac{1}{2}}\eta_{\frac{k}{2}}.
 \end{eqnarray*}
 Note that
 \begin{eqnarray*} h(0)v^{i,\frac{i}{2}+1}=-2v^{i,\frac{i}{2}+1},\ h(0)v^{i,\frac{i}{2}-1}=2v^{i,\frac{i}{2}-1}
 \end{eqnarray*}
 and $h^{'}(0)(v^{i,\frac{i}{2}+1}+v^{i,\frac{i}{2}-1})\in h^{'}(0)^{2}v^{i,\frac{i}{2}}$, we can deduce that
 \begin{eqnarray*} (e^{'}(0)-f^{'}(-1))\eta_{\frac{k}{2}}\in  (h(0)(v^{i,\frac{i}{2}+1}+v^{i,\frac{i}{2}-1}))_{m-\frac{1}{2}}\eta_{\frac{k}{2}}=h^{'}(0)(e^{'}(0)^{\frac{i}{2}-1}\eta_{i}+
 e^{'}(0)^{\frac{i}{2}+1}\eta_{i})_{m-\frac{1}{2}}\eta_{\frac{k}{2}}.
 \end{eqnarray*}
  This is a process from $i^{'}-\frac{i}{2}\in 2\Z+1$ to $i^{'}-\frac{i}{2}\in 2\Z$, which contradicts  (\ref{eq:.5..5}),  since $(e^{'}(0)-f^{'}(-1))\eta_{\frac{k}{2}}=(f-e)_{-\frac{1}{2}}\eta_{\frac{k}{2}}\in \widetilde{W(k,\frac{k}{2})}^{+}$. So $M=\widetilde{W(k,\frac{k}{2})}$.  (\ref{eq:5.31.}) and (\ref{eq:5.32.}) can be obtained following from the proof of (\ref{eq:5.29.}) and (\ref{eq:5.30.}).

For (\ref{eq:5.23}), from Theorem \ref{para-fusion-untwist}, we have $$(M^{\frac{k}{2},0})^{+}\boxtimes(M^{k,\frac{k}{2}})^{+}=(M^{\frac{k}{2},0})^{+}.$$
From (\ref{eq:5.16}), (\ref{eq:5.17}), we have that for $j\neq\frac{k}{2}$,
\[(M^{k,\frac{k}{2}})^{+}\boxtimes W(k,j)^{+}=\begin{cases}W(k,j)^{+},\ & \mbox{if} \ j\in 2{\mathbb{Z}}\cr
W(k,j)^{-},\ &\mbox{if} \  j\in 2{\mathbb{Z}}+1.\end{cases}\]
Since $j\in 2{\mathbb{Z}}+1$, we have
\begin{eqnarray*}\begin{split} (M^{\frac{k}{2},0})^{+}\boxtimes(M^{k,\frac{k}{2}})^{+}\boxtimes W(k,j)^{+}
&=(M^{\frac{k}{2},0})^{+}\boxtimes W(k,j)^{+}\\
&=(M^{\frac{k}{2},0})^{+}\boxtimes W(k,j)^{-}
\end{split}
\end{eqnarray*}
by associativity of the fusion product. Then from  Theorem \ref{quantum-dimension}, we have
\begin{eqnarray*}
 \mbox{qdim}_{K_{0}^{\sigma}}(M^{\frac{k}{2},0})^{+}=\frac{\sin\frac{\pi(\frac{k}{2}+1)}{k+2}}{\sin\frac{\pi}{k+2}},  \
\mbox{qdim}_{K_{0}^{\sigma}}W(k,j)^{+}
=\sqrt{k}\frac{\sin\frac{\pi(j+1)}{k+2}}{\sin\frac{\pi}{k+2}} \ \mbox{for} \ j\neq\frac{k}{2}.\end{eqnarray*}
By using
\begin{eqnarray*}\mbox{qdim}_{K_{0}^{\sigma}}\Big((M^{\frac{k}{2},0})^{+}\boxtimes W(k,j)^{+}\Big)=
\mbox{qdim}_{K_{0}^{\sigma}}(M^{\frac{k}{2},0})^{+}\cdot\mbox{qdim}_{K_{0}^{\sigma}}W(k,j)^{+},
\end{eqnarray*}
and noticing that $(M^{\frac{k}{2},0})^{+}\subseteq L(k,\frac{k}{2})^{+}$, $W(k,j)^{+}\subseteq \overline{L(k,j)}^{+}$, together with (\ref{eq:5.43.}) and $W(k,l)\cong W(k,k-l)$ as $K_0^{\sigma}$-module, we can deduce
(\ref{eq:5.23}).

For (\ref{eq:5.24}) and (\ref{eq:5.24'}), we prove the case $k\in4\Z+2$. If $k\in4\Z$, the proof is similar. Note that in this case $j\neq\frac{k}{2}$. Since $(M^{\frac{k}{2},0})^{+}\boxtimes(M^{k,\frac{k}{2}})^{+}=(M^{\frac{k}{2},0})^{+}$, we have
\begin{eqnarray}(M^{\frac{k}{2},0})^{+}\boxtimes W(k,j)^{+}=(M^{\frac{k}{2},0})^{+}\boxtimes(M^{k,\frac{k}{2}})^{+}\boxtimes W(k,j)^{+}. \label{eq:5.34.}
\end{eqnarray}
Then from Theorem \ref{quantum-dimension}, we have
\begin{eqnarray*}
 \mbox{qdim}_{K_{0}^{\sigma}}(M^{\frac{k}{2},0})^{+}=\frac{\sin\frac{\pi(\frac{k}{2}+1)}{k+2}}{\sin\frac{\pi}{k+2}},  \
\mbox{qdim}_{K_{0}^{\sigma}}W(k,j)^{+}
=\sqrt{k}\frac{\sin\frac{\pi(j+1)}{k+2}}{\sin\frac{\pi}{k+2}} \ \mbox{for} \ j\neq\frac{k}{2},\end{eqnarray*}
and
\begin{eqnarray*}\mbox{qdim}_{K_{0}^{\sigma}}\Big((M^{\frac{k}{2},0})^{+}\boxtimes W(k,j)^{+}\Big)=
\mbox{qdim}_{K_{0}^{\sigma}}(M^{\frac{k}{2},0})^{+}\cdot\mbox{qdim}_{K_{0}^{\sigma}}W(k,j)^{+}.
\end{eqnarray*}
Moreover, from (\ref{eq:5.19}), we have $$(M^{k,\frac{k}{2}})^{+}\boxtimes W(k,\frac{k}{2})^{+}=\widetilde{W(k,\frac{k}{2})}^{+},$$ and from (\ref{eq:5.20}), we have
$$(M^{k,\frac{k}{2}})^{+}\boxtimes \widetilde{W(k,\frac{k}{2})}^{+}=W(k,\frac{k}{2})^{+}.$$ From (\ref{eq:5.18}), we have
$$(M^{k,\frac{k}{2}})^{+}\boxtimes W(k,l)^{\pm}=W(k,k-l)^{\mp}$$ for $|\frac{k}{2}-j|\leq l <\frac{k}{2}$ and $l\in 2\Z+1$. Noticing that $(M^{\frac{k}{2},0})^{+}\subseteq L(k,\frac{k}{2})^{+}$, $W(k,j)^{+}\subseteq \overline{L(k,j)}^{+}$, together with (\ref{eq:5.43.}), we can deduce that (\ref{eq:5.34.}) implies (\ref{eq:5.24}), and then (\ref{eq:5.24'}) follows immediately.

For (\ref{eq:5.25}), from Theorem \ref{quantum-dimension}, the quantum dimension of the left side of the equation (\ref{eq:5.25}) is
\begin{eqnarray*}\begin{split} \mbox{qdim}_{K_{0}^{\sigma}}\Big((M^{\frac{k}{2},0})^{+}\boxtimes W(k,\frac{k}{2})^{+}\Big)
&=\mbox{qdim}_{K_{0}^{\sigma}}(M^{\frac{k}{2},0})^{+}\cdot\mbox{qdim}_{K_{0}^{\sigma}}W(k,\frac{k}{2})^{+}\\
&=\frac{\sqrt{k}}{2}\frac{\sin\frac{\pi(\frac{k}{2}+1)}{k+2}}{\sin\frac{\pi}{k+2}}\frac{\sin\frac{\pi(\frac{k}{2}+1)}{k+2}}{\sin\frac{\pi}{k+2}}.
\end{split}
\end{eqnarray*}
Since $(M^{\frac{k}{2},0})^{+}\subseteq L(k,\frac{k}{2})^{+}$, $W(k,\frac{k}{2})^{+}\subseteq \overline{L(k,\frac{k}{2})}^{+}$, by using (\ref{eq:5.43.}), and noticing that $l\neq \frac{k}{2}$ in this case, we can deduce (\ref{eq:5.25}), and (\ref{eq:5.25'}) follows immediately. Since by (\ref{eq:5.19}), $(M^{k,\frac{k}{2}})^{+}\boxtimes W(k,\frac{k}{2})^{+}=\widetilde{W(k,\frac{k}{2})}^{+}$, we have
$$(M^{\frac{k}{2},0})^{+}\boxtimes W(k,\frac{k}{2})^{+}=(M^{\frac{k}{2},0})^{+}\boxtimes(M^{k,\frac{k}{2}})^{+}\boxtimes W(k,\frac{k}{2})^{+}=(M^{\frac{k}{2},0})^{+}\boxtimes\widetilde{W(k,\frac{k}{2})}^{+}.$$
Thus (\ref{eq:5.26}) and (\ref{eq:5.26'}) follow immediately.

\end{proof}

\begin{thm}\label{thm:contragredient.para.}
 All the irreducible modules of the ${\mathbb{Z}}_{2}$-orbifold parafermion vertex operator algebra $K_{0}^{\sigma}$ are self-dual.
\end{thm}
\begin{proof}
From Theorem \ref{thm:orbifold3'} and Remark \ref{rmk:orbifold3'}, we know that the irreducible modules of $K_{0}^{\sigma}$ are twisted type modules and untwisted modules of type $I$ and type $II$, and the lowest weights of each irreducible $K_{0}^{\sigma}$-modules are listed in Proposition 3.13, Proposition 3.14 and Proposition 3.6 in \cite{JW}. Let $W$ be an irreducible $K_{0}^{\sigma}$-module. Since the top level of an irreducible $K_{0}^{\sigma}$-module $W$ is one-dimensional, set the top level $W_{0}=\mathbb{C}v$ and the top level of its contragredient modules $W_{0}^{'}=\mathbb{C}v^{'}$. Then $o(\omega)=\omega_{1}$ acts on the top level as scalar multiples. From the definition of the contragredient module (\ref{eq:3.0}), we have $$\langle o(\omega)v^{'}, v\rangle=\langle v^{'}, o(\omega)v\rangle.$$ It follows that $v$ and $v^{'}$ have the same weight. From Proposition 3.13 in \cite{JW}, we know that the lowest weights of irreducible twisted type modules of $K_{0}^{\sigma}$ are pairwise different, so the irreducible twisted type modules of $K_{0}^{\sigma}$ are self-dual. From Proposition 3.6 in \cite{JW}, we know that the lowest weights of irreducible untwisted $K_{0}^{\sigma}$-modules of type $II$ are pairwise different, thus the irreducible untwisted $K_{0}^{\sigma}$-modules of type $II$ are also self-dual. For the case of the irreducible untwisted $K_{0}^{\sigma}$-module of type $I$, since $(M^{k,0})^{+}=K_{0}^{\sigma}$, it is self-dual. If $k\in 2\mathbb{Z}$, we know that $\mathbf{1}\in K_{0}^{\sigma}\subseteq (M^{k,0})^{+}\boxtimes((M^{k,0})^{+})^{'}$, and
\begin{eqnarray*}v^{\frac{k}{2},0}=\sum\limits_{\tiny{\begin{split}I\subseteq\{1,2,\cdots,k\} \\  |I|=\frac{k}{2} \ \ \ \ \ \end{split}}}e^{\alpha_{I}/2} \in (M^{\frac{k}{2},0})^{+},\end{eqnarray*}
from (\ref{eq:3.4.}). Then we can deduce that
\begin{eqnarray*}v^{\frac{k}{2},\frac{k}{2}}=\sum\limits_{\tiny{\begin{split}J\subseteq\{1,2,\cdots,k\} \\  |J|=\frac{k}{2} \ \ \ \ \ \end{split}}}e^{-\alpha_{J}/2} \in ((M^{\frac{k}{2},0})^{+})^{'}.\end{eqnarray*}
Thus $((M^{\frac{k}{2},0})^{+})^{'}=(M^{\frac{k}{2},\frac{k}{2}})^{+}\cong(M^{\frac{k}{2},0})^{+}$, and so $(M^{\frac{k}{2},0})^{+}$ is self-dual. It follows that $(M^{\frac{k}{2},0})^{-}$ is self-dual.

If $i\in 2\mathbb{Z}$, we know that $\mathbf{1}\in K_{0}^{\sigma}\subseteq (M^{i,\frac{i}{2}})^{+}\boxtimes((M^{i,\frac{i}{2}})^{+})^{'}$, and note that from (\ref{eq:3.4.}),
\begin{eqnarray*}v^{i,\frac{i}{2}}=\sum\limits_{\tiny{\begin{split}I\subseteq\{1,2,\cdots,k\} \\  |I|=i \ \ \ \ \ \end{split}}}\sum\limits_{\tiny{\begin{split}J\subseteq I \\  |J|=\frac{i}{2} \end{split}}}e^{\alpha_{I-J}/2-\alpha_{J}/2} \in (M^{i,\frac{i}{2}})^{+}.\end{eqnarray*}
Then we can deduce that $v^{i,\frac{i}{2}}\in ((M^{i,\frac{i}{2}})^{+})^{'}$. Thus $((M^{i,\frac{i}{2}})^{+})^{'}=(M^{i,\frac{i}{2}})^{+}$, so $(M^{i,\frac{i}{2}})^{+}$ is self-dual. It follows that $(M^{i,\frac{i}{2}})^{-}$ is self-dual.
\end{proof}

\begin{rmk} From  Proposition \ref{fusionsymm.}, we know that Theorem \ref{para-fusion-untwist}, Theorem \ref{para-fusion-twist} and Theorem \ref{thm:contragredient.para.} give the fusion rules of all the irreducible $K_0^{\sigma}$-modules.
\end{rmk}

\begin{rmk} For $k=4$, from \cite{DLY2}, we know that the parafermion vertex operator algebra $K_0$ is isomorphic to the lattice vertex operator algebra $V_{{\Z}\alpha}^{+}$ with $\langle \alpha,\alpha \rangle=6$, so the orbifold $K_{0}^{\sigma}$ is isomorphic to the lattice vertex operator algebra $V_{{\Z}\alpha}^{+}$ with $\langle \alpha,\alpha \rangle=24$. In this case, our result about the  fusion rules of $K_{0}^{\sigma}$ is the same as that of the orbifold $V_{L}^{+}$ given in \cite{Ab1}.
\end{rmk}

\end{document}